\documentclass[reqno,a4paper,11pt]{amsart}  
\usepackage{enumerate,amsmath,amssymb}     
\usepackage{pstricks,pst-node,pst-coil,pst-plot}  
\usepackage{hyperref}
\theoremstyle{plain}      
 
\newtheorem{thm}{Theorem}[section]     
\newtheorem{theorem}[thm]{Theorem}     
     
\newtheorem{corollary}[thm]{Corollary}     
     
\newtheorem{lemma}[thm]{Lemma}     
     
\newtheorem{proposition}[thm]{Proposition}     
\newtheorem{conjecture}[thm]{Conjecture}     
\theoremstyle{remark}      
\newtheorem{example}[thm]{Example} 
\newtheorem{remark}[thm]{Remark}

\theoremstyle{definition}      
     
\newtheorem{definition}[thm]{Definition}     

\def\al{{\alpha}}

\let\La\Lambda

\def\Si{{\Sigma}}
\def\ga{{\gamma}}
\def\epsilon{{\varepsilon}}
\def\ep{{\varepsilon}}
\def\Ga{{\Gamma}}

\def\eg{e.\thinspace g.\ \ignorespaces}
\def\ie{i.\thinspace e.\ \ignorespaces}
       
\DeclareMathAlphabet{\doba}{U}{msb}{m}{n}         
\gdef\mC{\doba{C}}

\gdef\mN{\doba{N}}

\gdef\mR{\doba{R}}
\gdef\mS{\doba{S}}

\gdef\mZ{\doba{Z}}

\def\cR{\mathcal{R}}

\def\Hom{{\mathop{\rm Hom}}}
\def\End{{\mathop{\rm End}}}
\def\vol{{\mathop{\rm vol}}}
\def\id{{\mathop{\rm id}}}
\let\Id\id

\def\mmod{{\;\mathop{\rm mod}\,}}
\def\mmmod{{\mathop{\rm mod}\;}}

\def\ind{{\mathop{\rm ind}}}
\let\pa\partial     
     
\let\<\langle 
\let\>\rangle 
\let\ti\tilde
\let\witi\widetilde
\def\divv{{\mathop{\rm div}}}

\newcommand{\definedas}{\mathrel{\raise.095ex\hbox{\rm :}\mkern-5.2mu=}}

\def \be{\begin{eqnarray*}}
\def \ee{\end{eqnarray*}}
\def \ben{\begin{enumerate}}
\def \een{\end{enumerate}}
\def \beit{\begin{itemize}}
\def \eeit{\end{itemize}}
\def \bui#1#2{\mathrel{\mathop{\kern 0pt#1}\limits^{#2}}}
\def \buil#1#2{\mathrel{\mathop{\kern 0pt#1}\limits_{#2}}}
\def \bfll{\begin{flushleft}}
\def \efll{\end{flushleft}}
\def \bflr{\begin{flushright}}
\def \eflr{\end{flushright}}

\def \lra{\longrightarrow}
\def \lmt{\longmapsto}

\def \wih{\widehat}

\def \C{\mathbb{C}}
\def \R{\mathbb{R}}

\def\Dir{{\mathchoice{D\kern-7pt\slash\kern1.8pt}{D\kern-7pt\slash\kern1.8pt}{{\scriptstyle D\kern-5.4pt\slash\kern1.3pt}}{{\scriptscriptstyle D\kern-5.0pt\slash\kern1.1pt}}}}
\let\ol\overline
\let\wihat\widehat

\begin{document}     


\title
[Dirac-harmonic maps from index theory]
{Dirac-harmonic maps from index theory}
 
\author{Bernd Ammann} 
\address{Fakult\"at f\"ur Mathematik \\ 
Universit\"at Regensburg \\
D-93040 Regensburg}
\email{bernd.ammann@mathematik.uni-regensburg.de}

\author{Nicolas Ginoux} 
\address{Fakult\"at f\"ur Mathematik \\ 
Universit\"at Regensburg \\
D-93040 Regensburg}
\email{nicolas.ginoux@mathematik.uni-regensburg.de}

\begin{abstract}
We prove existence results for Dirac-harmonic maps using index theoretical tools.
They are mainly interesting if the source manifold has dimension $1$ or 
$2$ modulo $8$.
Our solutions are uncoupled in the sense that the underlying map between 
the source and target manifolds is a harmonic map. 
\end{abstract}

%

\date{\today}


\maketitle

\setcounter{tocdepth}{1}
\tableofcontents


\section{Introduction}

Dirac-harmonic maps are stationary points of the fermionic analogue
of the energy functional. The associated Euler-Lagrange equation is 
a coupled system consisting of a map $f:M\to N$ and a spinor on $M$
twisted by $f^*TN$, such that the (twisted) 
spinor is in the kernel of the (twisted) Dirac operator. 
The goal of this article is to use index theoretical tools
for providing such Dirac-harmonic maps. 

\subsection{Dirac-harmonic maps}\label{subsec.dhintro}

In mathematical physics a (non-linear) sigma-model consists of Riemannian
manifolds $M$ and $N$. For simplicity we always assume $M$ and $N$ 
to be compact. The classical or bosonic energy of a map $f:M\to N$ is defined as
$\frac12 \int_M|df|^2$.
Stationary points of
this functional are called harmonic maps. 
Harmonic maps from $M$ to $N$ 
are interpreted as solutions of this sigma-model. The term ``non-linear'' indicates that the target manifold 
$N$ is not a vector space and will be omitted from now on. 

A super-symmetric sigma-model 
(see \cite[Chapter 3]{deligne.freed:99}, \cite{Freed5lect}) 
also consists of Riemannian
manifolds~$M$ and~$N$, but in addition to above one assumes that~$M$ 
carries a fixed spin structure.
A fermionic energy functional can now be 
defined on pairs $(f,\Phi)$, where $f$ is again a map $M\to N$ 
and $\Phi$ is a spinor
on $M$ twisted by $f^*TN$.
In the literature a linear and a non-linear versions of this functional are considered, see 
Subsection~\ref{subsec.variants} for details. 
In our article and in most of the mathematical literature one considers
the fermionic energy functional 
  $$L(f,\Phi)=\frac12 \int_M \left(|df|^2+\<\Phi,\Dir^f \Phi\>\right)\,dv^M,$$
where $\Dir^f$ is the Dirac operator acting on spinors twisted by $f^*TN$.
Stationary points $(f_0,\Phi_0)$ of this fermionic energy functional are called 
Di\-rac-har\-mo\-nic maps.
They are characterized by the equations $\Dir^{f_0}\Phi_0=0$ and $\mathrm{tr}_g(\nabla df_0)=\frac{V_{\Phi_0}}{2}$, 
see Proposition \ref{p:critpointL1} below. 
They are interpreted as solutions of the 
super-symmetric sigma-model.
Regularity and existence questions for 
Dirac-harmonic maps have been the subject of many articles recently where important results have been obtained, see Section~\ref{sec.previous}. 

\subsection{Trivial Dirac-harmonic maps}\label{subsec.triv}

A harmonic map together with a spinor 
that vanishes everywhere provides a Dirac-harmonic map.
In the present article, such Dirac-harmonic maps are called 
\emph{spinor-trivial} Dirac-harmonic maps. 

Another class
of obvious solutions consists of constant maps $f$ together with a spinor 
$\Phi$ lying in the kernel of the (untwisted) Dirac operator, see Subsection \ref{subsec.kernel}. 
Such Dirac-harmonic maps will be called \emph{map-trivial}.
However, the 
dimension of the kernel depends in a subtle way on the 
conformal structure of~$M$. For the sake of clarity we give a small review 
on known results about the kernel in Subsection~\ref{subsec.kernel}.

\subsection{Main result}\label{subsec.main}
\begin{definition}
A Dirac-harmonic map will be called 
\emph{uncoupled} if $f$ is harmonic, otherwise it is called \emph{coupled}.
\end{definition}
Most of the existence results for Dirac-harmonic maps in the 
literature actually yield uncoupled Dirac-harmonic maps. 
Coupled Dirac-harmonic maps are discussed in \cite[Thm. 3]{JostMoZhu09}, 
see Subsection~\ref{subsec.coupleddh} for details.

Our result provides criteria for the existence of sufficiently many 
non-trivial, uncoupled Dirac-harmonic maps.

\begin{theorem}\label{t:mainresult}
For $M$ and $N$ as above consider a homotopy class $[f]$ of maps $f:M^m\to N^n$ such that the 
$\mathrm{KO}_m({\rm pt})$-valued index $\al(M,[f])$ is non-trivial. 
Assume that $f_0\in [f]$ is a harmonic map. Then there is a linear
space~$V$ of real dimension $a_m$ such that all 
$(f_0,\Phi)$, $\Phi\in V$, are Dirac-harmonic maps, where 
\begin{equation*}a_m:=
\left\{ 
\begin{matrix}
2&\mbox{if }m\equiv 0 \; (8)\hfill\\
2&\mbox{if }m\equiv 1 \; (8)\hfill\\
4&\mbox{if }m\equiv 2 \; (8)\hfill\\
4&\mbox{if }m\equiv 4 \; (8).\hfill
\end{matrix}\right.
\end{equation*}
\end{theorem}

This theorem 
is an immediate connsequence of Theorem~\ref{t:f0notpertmin}. 
It is mainly interesting in dimensions~$1$ and~$2$ modulo~$8$.
In dimensions divisible by $4$ the statement
directly follows from the Atiyah-Singer-index theorem and the grading techniques of 
Sections~\ref{sec.grad} and~\ref{sec.m4k}.
In dimension $1$ modulo $8$, these grading techniques would also yield a similar result, but the space of Dirac-harmonic maps obtained
this way would have only real dimension $1$ instead of $2$.

In order to derive such results we have structured the article as follows.
We start by giving an overview about known results in Section~\ref{sec.previous} in order to compare 
them to our results. It is logically independent from the rest.
Having introduced the notations and the fermionic energy in
Section~\ref{sec.prelim}, we recall
some well-known facts about the index 
of twisted Dirac operators in Section~\ref{sec.index} and the variational formula in Section~\ref{sec.var}.
This immediately leads to the first existence result using grading techniques (Section~\ref{sec.grad})
and the result in dimensions $m$ divisible by $4$ (Section~\ref{sec.m4k}).
In order to get better results in dimensions $8k+1$ and $8k+2$ we study the ``minimal'' and ``non-minimal case'' 
in Sections~\ref{sec.min} and \ref{sec.nonmin}. 
Finally, we discuss in Section~\ref{sec.examples} numerous examples to which the above theorem applies.

\section{Previously known results}\label{sec.previous}

To compare our results with the existing literature
we briefly recall several known facts about Dirac-harmonic maps.

Results about regularity,
investigations about the necessity for the mapping component $f$ to be harmonic,
removal of singularity theorems and  blow-up analysis of Dirac-harmonic maps
are developed in \cite{ChenJostLiWang06,ChenJostLiWang05,Zhao07,
WangXu09,Zhu_CVPDE09,Zhu_AGAG09}. We do not give details, as these 
issues are far from our topic.

We now summarize existence results for Dirac-harmonic maps. 
Not many concrete examples are known. Apart from the trivial examples 
in subsections~\ref{subsec.triv} and~\ref{subsec.kernel}, almost all examples occur in dimension $m=2$ and can be divided into 
uncoupled and coupled solutions. 
Before we describe them, recall first that, in dimension $m=2$, the existence of a Dirac-harmonic map only depends 
on the conformal class of the metric chosen on $M$: a pair $(f_0,\Phi_0)$ 
is Dirac-harmonic on $(M^2,g)$ if and only if $(f_0,e^{-\frac{u}{2}}\Phi_0)$ 
is Dirac-harmonic on $(M^2,e^{2u}g)$, whatever $u\in C^\infty(M,\R)$ is.
{}From the point of view of string theory, the case $m=2$ is 
of central importance as it describes the evolution of strings in space-times.

\subsection{Existence of uncoupled Dirac-harmonic maps}

The first existence result, which appears 
in \cite[Prop. 2.2]{ChenJostLiWang06} for $M=N=\mathbb{S}^2$ and 
then in \cite[Thm. 2]{JostMoZhu09} for general surfaces $M$, is based on an 
explicit construction involving a harmonic map and 
so-called \emph{twistor-spinors} on the source manifold.
Twistor-spinors are sections of the spinor bundle $\Sigma M$ of $M$ 
lying in the kernel of the Penrose operator 
$P:C^\infty(M,\Sigma M)\lra C^\infty(M,T^*M\otimes\Sigma M)$, 
$P_X\varphi:=\nabla_X\varphi+\frac{1}{m}X\cdot \Dir\varphi$ (here $\Dir$ denotes the untwisted Dirac operator on $M$). 
The construction goes as follows.
Let $M$ be any $m$-dimensional (non ne\-ces\-sa\-ri\-ly compact) spin manifold and $f:M\lra N$ be a smooth map.
For an arbitrary $\varphi\in C^\infty(M,\Sigma M)$ consider the smooth section $\Phi:=\sum_{j=1}^m e_j\cdot\varphi\otimes df(e_j)$ of $\Sigma M\otimes f^*TN$, where $(e_j)_{1\leq j\leq m}$ is a local orthonormal basis of $TM$.
Note that this (pointwise) definition does not depend on the choice of the local orthonormal basis $(e_j)_{1\leq j\leq m}$ and that the vanishing of $\Phi$ is equivalent to that of $\varphi$ as soon as $f$ is an immersion.
Then a short computation shows that
\[\Dir^f\Phi=-2\sum_{j=1}^m P_{e_j}\varphi\otimes df(e_j)+\frac{2-m}{m}\sum_{j=1}^me_j\cdot \Dir\varphi\otimes df(e_j)-\varphi\otimes\mathrm{tr}_g(\nabla df).\]
Therefore, if $\varphi$ is a twistor-spinor, $m=2$ and $f$ is harmonic, then $\Dir^f\Phi=0$.
Moreover, if $m=2$, then one calculates that the vector field $V_\Phi$ defined in Proposition~\ref{p:critpointL1} vanishes.
Therefore, $(f,\Phi)$ is a Dirac-harmonic map as soon as $f$ is 
harmonic and the spinor $\varphi$ involved in the definition of $\Phi$ is a 
twistor-spinor.
This construction has the obvious drawback that the only closed surfaces 
admitting non-zero twistor-spinors are the $2$-torus $\mathbb{T}^2$ 
with its non-bounding spin structure and the $2$-sphere $\mathbb{S}^2$.
In particular no example on a 
hyperbolic surface can be given using this approach.
However, let us mention that all Dirac-harmonic maps from $\mathbb{S}^2$ 
to $\mathbb{S}^2$ are of that form \cite[Thm. 1.2]{Yang09}.
Vanishing results in case the spinor part $\Phi$ has the form above are considered 
in \cite{Mo10}.

\subsection{Existence of coupled Dirac-harmonic maps}\label{subsec.coupleddh}

Up to the know\-led\-ge of the authors, only two examples of coupled Dirac-harmonic maps are known.
Both are constructed explicitly considering particular iso\-me\-tric but 
non-minimal immersions from special non-compact source manifolds into the hyperbolic space.
The first one deals with an explicit non-minimal isoparametric immersion from a hyperbolic surface of revolution 
into the $3$-di\-men\-sio\-nal hyperbolic space \cite[Thm. 3]{JostMoZhu09}.

The second one is provided by the totally umbilical (but non-totally geodesic) embedding of hyperbolic hyperplanes of sectional curvature $-\frac{4}{m+2}$ into the $m+1(\geq 4)$-dimensional hyperbolic space of sectional curvature $-1$. This example can be handled with the twistor-spinor-ansatz of~\cite{JostMoZhu09} and is carried out in~\cite{AmmGinJostMoZhu09}. 

\subsection{More on trivial Dirac-harmonic maps}\label{subsec.kernel}
In this subsection we will explain a product construction for producing examples of Dirac-harmonic maps, and then we will summarize
the knowledge about map-trivial ones. The main objective of the article is to derive existence of Dirac-harmonic maps which do not fall in any of these trivial categories: spinor-trivial one, map-trivial ones, products of them.

To understand the product construction, assume that $f_1:M\to N_1$ and  $f_2:M\to N_2$ are smooth maps, so that they are the components of the map $(f_1,f_2):M\to N_1\times N_2$.
Similarly, let $\Phi_j\in \Ga(\Si M \otimes f_j^*TN_j)$, $j=1,2$, and let $(\Phi_1,\Phi_2)\in \Ga(\Si M \otimes (f_1^*TN_1\oplus f_2^*TN_2)$
denote their sum. 
Let $L$ (resp.\ $L_i$) be the fermionic energy functional for $N=N_1\times N_2$ (resp.\ $N=N_i$).
Then $L((f_1,f_2),(\Phi_1,\Phi_2))= L_1(f_1,\Phi_1) + L_2(f_2,\Phi_2)$.
Thus, the map $((f_1,f_2),(\Phi_1,\Phi_2))$ is Dirac-harmonic if and only if both $(f_1,\Phi_1)$ and $(f_2,\Phi_2)$ are Dirac-harmonic.

So we obtain new examples of Dirac-harmonic maps by taking products of map-trivial Dirac harmonic maps 
with spinor-trivial maps.

The pair $(f,\Phi)$ is a map-trivial Dirac-harmonic map if and only if the following holds:
\begin{enumerate}
 \item[(1)] $f$ is constant, say $f\equiv x\in N$,
 \item[(2)] $\Phi\in \Gamma(\Si M\otimes T_xN)
\cong  \Gamma(\Si M)\otimes \mR^n$, $n=\dim N$, is a harmonic spinor, i.e.\ $\Dir^f \Phi=0$ and $\Dir^f=\Dir \otimes \Id_{\mR^n}$ 
where $\Dir :\Gamma(\Si M)\to \Gamma(\Si M)$ is the untwisted Dirac operator.
\end{enumerate}

Thus, to determine the map-trivial Dirac-harmonic maps, one has to determine the kernel of the untwisted Dirac operator.
The dimension of the kernel is invariant under conformal change of the metric, but it depends on the conformal class. Important 
progress was obtained in particular in \cite{hitchin:74}, \cite{baer:96b}, \cite{baer:97b} and \cite{baer.dahl:02}.
The $\mathrm{KO}_m({\rm pt})$-valued index obviously gives a lower bound on the dimension of the kernel. We say, that a Riemannian metric~$g$ is 
\emph{$\Dir$-minimal} if this lower bound is attained.

At first one knows that generic metrics are $\Dir$-minimal.
\begin{theorem}[\cite{ammann.dahl.humbert:09}]
Let $M$ be a connected compact spin manifold. Then the set of $\Dir$-minimal Riemannian metrics is open and dense in the set of all Riemannian metrics in the $C^k$-topology for any $k\in\{1,2,3,\ldots\}\cup\{\infty\}$.
\end{theorem}

The theorem was already proved in~\cite{baer.dahl:02} for simply connected manifolds of dimension at least $5$. There is also a stronger 
version of the theorem in \cite{ammann.dahl.humbert:p09} which states that any metric can be perturbed in an arbitrarily small 
neighborhood to a $\Dir$-minimal one.

In the case $m=2$ the complex dimension of the kernel of $\Dir$ is at most $g+1$ where $g$ is the genus 
of the surface $M$ \cite[Prop.~2.3]{hitchin:74}. This upper bound determines the dimension of the kernel of $\Dir$ in the case $g\leq 2$, $\al(M)=0$ 
and in the case $g\leq 4$, $\al(M)=1$. The dimension of the kernel can also be explicitly determined for hyperelliptic surfaces
\cite{baer.schmutz:92}. In the same article the dimension of the kernel is calculated for special non-hyperelliptic surfaces of genus $4$ and $6$.
However, for general surfaces of sufficiently large genus, the dimension of the kernel is unknown.

In contrast to $m=2$, it is conjectured that for $m\geq 3$ there is no topological quantity that yields an upper bound on $\dim \ker \Dir$.
\begin{conjecture}\cite[p.~941]{baer:96b}
On any non-empty compact spin manifold of dimension $m\geq 3$, there is a sequence of Riemannian metrics~$g_i$
with 
  $$\lim_{i\to \infty}\dim\ker\Dir^{g_i}=\infty.$$  
\end{conjecture}
The conjecture is known only in few cases as e.g. the $3$-sphere $S^3$ where Hitchin \cite{hitchin:74}
has shown that a sequence of Berger metrics
provides such a sequence of metrics.

In particular, the conjecture would include that any manifold with $m\geq 3$ carries a non-$\Dir$-minimal metric.
The latter statement is partially known to be true. It was shown in \cite{hitchin:74} for $m\equiv 7,0,1\;(8)$ and in 
\cite{baer:96b} for $m\equiv 3,7\;(8)$ that any compact $m$-dimensional spin manifold carries a metric with $\ker\Dir\neq 0$.
A similar statement is known for all spheres of dimension~$m\geq 4$ \cite{seegerdiss} and \cite{dahl:08}.

However it is still
unknown whether an arbitrary compact spin manifold of dimension $m\equiv 2,4,5,6\;(8)$ admits a non-$\Dir$-minimal metric.

\section{Notation and Preliminaries}\label{sec.prelim}

\subsection{Notation and Conventions}

The Riemannian manifolds $M$ and $N$ are always assumed to be compact 
without boundary, $m=\dim M$, $n= \dim N$. We assume that $M$ is spin
and --- unless stated otherwise ---
we always assume that one of the spin structures on $M$ is fixed. 
The (untwisted) spinor bundle on $M$ is denoted by $\Si M$, 
the fiber over $x\in M$ is denoted by $\Si_x M$. In this article
we consider complex spinors, \ie, $\Si M$ is a complex vector bundle
of rank $2^{[\frac{m}{2}]}$; similar statements also hold for real spinors.
The untwisted spinor bundle carries a Hermitian metric, 
a metric connection and a Clifford multiplication. 
For an introduction to these
structures we refer to textbooks on spin geometry such as
\cite{lawson.michelsohn:89,friedrich:00,hijazi:01}. These structures
allow to define an (untwisted) Dirac operator 
$\Dir:C^\infty(M,\Sigma M)\lra C^\infty(M,\Sigma M)$.

If $E$ is a real vector bundle with given metric and metric connection, we 
define the twisted spinor bundle as $\Si M\otimes E$, 
the tensor bundle being taken over~$\mR$.
The twisted spinor bundle carries similar structures as the untwisted one.
Sections of  $\Si M\otimes E$ will be called \emph{spinors} or 
\emph{spinors twisted by $E$}.
One uses the structures on  $\Si M\otimes E$ to define the twisted 
Dirac operator 
  $$\Dir^E:C^\infty(M,\Sigma M\otimes E)\lra C^\infty(M,\Sigma M\otimes E).$$ 
It is a first order elliptic and 
selfadjoint differential operator. 
Its spectrum is thus discrete, real and of finite multiplicity.

In this article $E$ will often be obtained as $E=f^*TN$ where $f:M\to N$ is 
smooth. Here $E$ is equipped with the pull-back of the metric 
and the Levi-Civita connection 
on~$TN$. In this case we simply write $\Dir^f$ for $\Dir^{f^*TN}$.

Our convention for curvature tensors is $R_{X,Y}^N=[\nabla_X^N,\nabla_Y^N]-\nabla_{[X,Y]}^N$, and $R$ is then considered as an element in $\Hom(TN\otimes TN\otimes TN,TN)$ through $R(X\otimes Y\otimes Z)=R(X,Y)Z$.

The sphere $S^n$ with its standard metric (of constant sectional curvature $1$) will be denoted as $\mS^n$. 

The complex projective space $\mC \mathrm{P}^n$ will be always equipped with its Fubini-Study metric of holomorphic
sectional curvature $4$. The tautological bundle on $\mC \mathrm{P}^n$ is denoted by $\ga_n$.

\subsection{Linear and non-linear energy functional}\label{subsec.variants}

As mentioned in the introduction, the fermionic energy functional is discussed in a linear and non-linear version 
in the literature. We will give some details now. For a concise introduction see~\cite{Freed5lect} or \cite[Chapter~3]{deligne.freed:99}.
 
In order to derive in terms of physically motivated conclusions the appropriate fermionic energy, one reformulates 
the classical energy using the structure of the category of all vector spaces. Then one replaces the category of vector spaces
by the category of super-vector spaces. 
One obtains a functional
  $$L_c(f,\Phi):=\frac12 \int_M (|df|^2 + \<\Phi,\Dir^f\Phi\> + \frac16 
\<\Phi,\cR (\Phi,\Phi)\Phi\>)\,dv^M.$$
Here $\cR$ is a section of 
$\Hom_\mC((\Si M\otimes_\mR TN)\otimes_\mC (\overline{\Si M}\otimes_\mR TN)\otimes_\mC (\Si M\otimes_\mR TN),\Si M\otimes_\mR TN)$ obtained by tensoring the Riemann curvature tensor
$R\in \Hom_\mR(TN\otimes_\mR TN\otimes_\mR TN,TN)$ with the spinorial contraction map
$\id\otimes \<\,.\,\>\in \Hom_\mC(\Si M \otimes_\mC \overline{\Si M}\otimes_\mC \Si M,\Si M)$, $\id\otimes \<\,.\,\>(\varphi_1,\varphi_2,\varphi_3)=\<\varphi_2,\varphi_3\>\varphi_1$.
This functional is often called the fermionic energy functional with curvature term, and stationary points of this functional are called 
\emph{Dirac-harmonic maps with curvature term}. However this functional makes analytical con\-si\-de\-ra\-tions involved, and thus
 only few analytical articles include this curvature term, e.\thinspace g.\ \cite{chen.jost.wang:p07a} and \cite{ChenJostWang07}. 
In \cite{isobe:11} interesting techniques are developed which might turn helpful to find solutions of the non-linear equation.

As the curvature term is of fourth order in $\Phi$ whereas the dominating term is quadratic on $\Phi$, it seems acceptable from the view point of physical applications to neglect the curvature term. The fermionic energy functional thus obtained
  $$L(f,\Phi):= \frac12 \int_M (|df|^2 + \<\Phi,\Dir^f\Phi\>)\,dv^M\in \mR,$$
is analytically much easier to study.
A pair $(f_0,\Phi_0)$ is a stationary point of $L$ if and only if it satisfies 
$\Dir^{f_0}\Phi_0=0$ and $\mathrm{tr}_g(\nabla df_0)=\frac{V_{\Phi_0}}{2}$, 
see Proposition \ref{p:critpointL1} below. Stationary points of this 
functional are called \emph{Dirac-harmonic maps}.
Starting with \cite{ChenJostLiWang06}, this functional has been intensively studied in the literature, see Subsection~\ref{sec.previous}.

The method developed in the present article applies to this linear version.

\section{Index theory of Dirac-harmonic maps}\label{sec.index}

Let $M$ be a closed $m$-dimensional Riemannian spin manifold with spin structure denoted by $\chi$.
Let $E\lra M$ be a Riemannian (real) vector bundle with metric connection.
Then one can associate to the twisted Dirac operator $\Dir^E:C^\infty(M,\Sigma M\otimes E)\lra C^\infty(M,\Sigma M\otimes E)$ an index 
$\al(M,\chi,E)\in \mathrm{KO}_m(\mathrm{pt})$ (see \eg \cite[p.151]{lawson.michelsohn:89}).
Using the isomorphism \cite[p.141]{lawson.michelsohn:89}
\begin{equation}\label{eq:KOm}
\mathrm{KO}_m(\mathrm{pt})\cong \left\{\begin{array}{ll}\mathbb{Z}&\textrm{ if }m\equiv 0\;(4)\\ \mathbb{Z}_2&\textrm{ if }m\equiv 1,2\;(8)\\0&\textrm{ otherwise,}\end{array}\right. 
\end{equation}
the index $\al(M,\chi,E)$ will be identified either with an integer 
or an element in the group
$\mathbb{Z}_2$ of integers modulo $2$.
We also say that $\al(M,\chi,E)$ is the \emph{$\al$-genus} of $E\lra M$.

The $\al$-genus can be easily determined out of $\ker \Dir^E$ using 
the following formula \cite[Thm. II.7.13]{lawson.michelsohn:89}:
\[\al(M,\chi,E)=\left\{\begin{array}{ll} 
\{\mathrm{ch}(E)\cdot\wih{A}(TM)\}[M]&\textrm{ if }m\equiv 0\;(8).\\ 
\left[\mathrm{dim}_{\C}(\mathrm{ker}(\Dir^E))\right]_{\mathbb{Z}_2}&\textrm{ if }m\equiv 1\;(8)\\ 
{}[\frac{\mathrm{dim}_{\C}(\mathrm{ker}(\Dir^E))}{2}]_{\mathbb{Z}_2}&\textrm{ if }m\equiv 2\;(8)\\ 
\frac{1}{2}\{\mathrm{ch}(E)\cdot\wih{A}(TM)\}[M]&\textrm{ if }m\equiv 4\;(8)\\ 
\end{array}\right.\]

As usual $\mathrm{ch}(E)$ is to be understood as the Chern character of $E\otimes_\R\C$.
In dimensions $1$ and $2$ modulo $8$ the $\al$-genus depends on the spin 
structure~$\chi$ on~$M$, however we often simply 
write $\al(M,E)$, when $\chi$ is clear from
the context. We will assume from now on and until the end of 
Section~\ref{sec.nonmin} 
that $M$ comes with a fixed orientation and spin structure, so we 
omit the notation $\chi$ in those sections.

The number $\al(M,E)$ is a spin-bordism-invariant, 
where a spin-bordism for manifold with vector bundles means 
that the restriction of a vector bundle to the boundary of the spin manifold must coincide 
with the vector bundle given on the boun\-da\-ry: namely, 
\cite[Thm. II.7.14]{lawson.michelsohn:89} implies that $\al$ defines 
a map $\Omega_*^{\mathrm{Spin}}(\mathrm{BO})\lra\mathrm{KO}_*(\mathrm{pt})$.
Note in particular that $\al(M,E)$ does depend neither on the metric 
nor on the connection chosen on $E\lra M$.
In case that $E$ is the trivial real line bundle 
$\underline{\R}:=\R\times M\lra M$, the number $\al(M,E)=:\al(M)$ 
is the classical 
$\al$-genus of $M$.\\

\begin{definition}\label{d:alpha0}
Let $M$ be a closed $m$-dimensional Riemannian spin manifold and $f:M\to N$ be a smooth map into an $n$-dimensional Riemannian manifold.
The \emph{$\al$-genus} of $f$ is $\al(M,f):=\al(M,f^*TN)$.
\end{definition}

The spin-bordism-invariance of the $\al$-genus has an important consequence for $\al(M,f)$.
We first make a

\begin{definition}\label{d:spinbordantmaps}
With the notations of {\rm Definition \ref{d:alpha0}}, two maps $f_1:M_1\to N$ and $f_2:M_2\to N$ are \emph{spin-bordant} in $N$,
if there is a spin manifold $W$ of dimension $m+1$ together with a map $F:W\to N$
such that $\pa W= -M_1\amalg M_2$ (in the sense of manifolds with spin structure) and such that 
$F_{|_{M_j}}=f_j$ for both $j=1,2$.
\end{definition}

Obviously, given any smooth spin-bordant maps $f_j:M_j\lra N$, $j=1,2$, the pull-back bundles $f_j^*(TN)\lra M_j$ are spin-bordant as vector bundles.
Therefore, we obtain the 

\begin{proposition}\label{p:alphaspinbordisminv}
Assume that $f_1:M_1\to N$ and $f_2:M_2\to N$ are spin-bordant in $N$.
Then $\al(M_1,f_1)=\al(M_2,f_2)$.
\end{proposition}

Since a homotopy between maps $f_j:M\lra N$ defines a spin-bordism between them, we deduce the

\begin{corollary}\label{c:alphahomotinv}
Assume that $f_1:M\to N$ and $f_2:M\to N$ are homotopic maps.
Then  $\al(M,f_1)=\al(M,f_2)$.
\end{corollary}

Because of this homotopy invariance we also write $\al(M,[f])$ for 
$\al(M,f)$, where $[f]$ is the homotopy class of $f$.

\section{Variational formulas for the fermionic energy functional}\label{sec.var}

We recall the following formulas for the first variation of $L$ \cite[Prop. 2.1]{ChenJostLiWang06}.

\begin{proposition}\label{p:critpointL1}
For $(f_0,\Phi_0)\in C^\infty(M,N)\times C^\infty(M,\Sigma M\otimes f_0^*TN)$ and some $\varepsilon>0$ let $(f_t,\Phi_t)_{t\in]-\varepsilon,\varepsilon[}$ be a variation of $(f_0,\Phi_0)$ which is differentiable at $t=0$. 
Denote by $\Dir_t:=\Dir^{f_t}$ for all $t\in]-\varepsilon,\varepsilon[$.
Then
\begin{equation}\label{eq:dDirdt}\frac{d}{dt}\langle\Dir_t\Phi_t,\Phi_t\rangle_{|_{t=0}}=h(V_{\Phi_0},\frac{\partial f}{\partial t}(0))+\langle\Dir_0\frac{\partial\Phi}{\partial t}(0),\Phi_0\rangle+\langle\Dir_0\Phi_0,\frac{\partial\Phi}{\partial t}(0)\rangle\end{equation}
where $h(V_{\Phi_0},Y)=\sum_{j=1}^m\langle e_j\cdot\otimes R_{Y,df_0(e_j)}^N\Phi_0,\Phi_0\rangle$ for all $Y\in TN$ with basepoint $f(x)\in N$
and every orthonormal frame $(e_j)_{1\leq j\leq m}$ of $T_xM$.
In par\-ti\-cu\-lar,
\[\frac{d}{dt}L(f_t,\Phi_t)_{|_{t=0}}=\int_M \Re e(\langle\Dir_0\Phi_0,\frac{\partial\Phi}{\partial t}(0)\rangle)-h\left(\mathrm{tr}_g(\nabla df_0)-\frac{V_{\Phi_0}}{2},\frac{\partial f}{\partial t}(0)\right)\,dv^M.\]
\end{proposition}

The linear endomorphism  
$e_j\cdot\otimes R_{Y,df_0(e_j)}^N\in \End(\Si_xM\otimes T_{f(x)}N)$
is obtained by tensoring Clifford multiplication $\varphi\mapsto e_j\cdot \varphi$
with $R_{Y,df_0(e_j)}^N\in \End(T_{f(x)}N)$. It is thus the real tensor product of a skew-hermitian endomorphism
with a skew-symmetric endomorphism, and thus a hermitian endomorphism. In particular the expression
$\langle e_j\cdot\otimes R_{Y,df_0(e_j)}^N\Phi_0,\Phi_0\rangle$ is real.
 
The dif\-fe\-ren\-tia\-bi\-li\-ty of $t\mapsto\Phi_t$ at $t=0$ is to be understood as follows.
Denote by $f:]-\varepsilon,\varepsilon[\times M\lra N$, $(t,x)\mapsto f(t,x):=f_t(x)$ the variation above.
The metric and the Levi-Civita connection on $TN$ induce a metric and a metric connection on the pull-back bundle $f^*TN\lra]-\varepsilon,\varepsilon[\times M$.
For every $t\in]-\varepsilon,\varepsilon[$ let $\beta_t:C^\infty(M,\Sigma M\otimes f_0^*TN)\lra C^\infty(M,\Sigma M\otimes f_t^*TN)$ be the unitary and pa\-ral\-lel isomorphism induced by the pa\-ral\-lel transport along the curves $[0,t]\rightarrow ]-\varepsilon,\varepsilon[\times M$, $s\mapsto (s,x)$, where $x$ runs in $M$.
Then we require the map $t\mapsto\wih{\Phi_t}:=\beta_t^{-1}\circ\Phi_t\in C^\infty(M,\Sigma M\otimes f_0^*TN)$ to be differentiable at $t=0$ in the following sense: the map $]-\varepsilon,\varepsilon[\lra C^\infty(M,\Sigma M\otimes f_0^*TN)$, $t\mapsto\wih{\Phi}_t$, has a derivative at $t=0$ which is at least continuous on $M$.
Here we consider the topology induced by the H$^{1,2}$-norm on $C^\infty(M,\Sigma M\otimes f_0^*TN)$.
In that case, we denote by $\frac{\partial\Phi}{\partial t}(0):=\frac{\partial\wih{\Phi}}{\partial t}(0)\in C^1(M,\Sigma M\otimes f_0^*TN)$.
Note that, if $\wih{\Dir}_t:=\beta_t^{-1}\circ\Dir_t\circ\beta_t:C^\infty(M,\Sigma M\otimes f_0^*TN)\lra C^\infty(M,\Sigma M\otimes f_0^*TN)$, then $\langle\wih{\Dir}_t\wih{\Phi}_t,\wih{\Phi}_t\rangle=\langle\Dir_t\Phi_t,\Phi_t\rangle$ since $\beta_t$ is unitary.\\

As a straightforward consequence of Proposition \ref{p:critpointL1}, we obtain:

\begin{corollary}\label{c:suffcondcritpointL1}
Let $f_0\in C^\infty(M,N)$ be a harmonic map and $\Phi_0\in\mathrm{ker}(\Dir^{f_0})$.
Assume the existence, for every smooth variation $(f_t)_{t\in]-\varepsilon,\varepsilon[}$ of $f_0$, of a variation $(\Phi_t)_{t\in]-\varepsilon,\varepsilon[}$ of $\Phi_0$ which is differentiable at $t=0$ and such that $\frac{d}{dt}\left(\Dir^{f_t}\Phi_t,\Phi_t\right)_{\mathrm{L}^2}{}_{|_{t=0}}=0$.
Then the pair $(f_0,\Phi_0)$ is an uncoupled Dirac-harmonic map.
\end{corollary}

\begin{proof}
Only $V_{\Phi_0}=0$ has to be proved, which follows from integrating (\ref{eq:dDirdt}) and using the self-adjointness of $\Dir^{f_0}$.
\end{proof}

\section{The graded case}\label{sec.grad}

In this section we consider the situation where the bundle $\Sigma M$ 
admits an orthogonal and parallel $\mathbb{Z}_2$-grading 
$G\in\End(\Sigma M)$, $G^2=\id$, anti-commuting 
with the Clifford-multiplication by tangent vectors.
Such a grading naturally exists
in even dimensions and in dimension $m\equiv1\;(8)$ and induces a grading of $\Sigma M\otimes f_0^*TN$.
 
In even dimension $m$, it is given by the Clifford action of the 
so-called complex volume form of $TM$, namely
$G(\varphi)=i^{\frac{m}{2}} e_1\cdot \ldots \cdot e_m\cdot \varphi$ for a positively 
oriented orthonormal local frame. One easily checks that the definition
of $G$ does not depend on the choice of the local frame, and thus $G$ is globally defined.
The spinor bundle then decomposes into two 
complex subbundles $\Sigma_+M$ and $\Sigma_-M$ associated 
to the $+1$ and $-1$-eigenvalue of $G$ respectively.
As $G$ is Hermitian and parallel, the decomposition  $\Sigma M= \Sigma_+M\oplus \Sigma_-M$
is orthogonal in the complex sense and parallel.

In dimension $m\equiv 1\;(8)$, the grading $G$ is 
provided by a real structure 
on the complex spinor representation, see \eg \cite[Sec. 1.7]{friedrich:00}. 
This map is complex anti-linear, involutive and anticommutes with the
Clifford action by tangent vectors. It is still self-adjoint in the real sense, \ie
$\Re e(\<G(\varphi),\psi\>)=\Re e(\<\varphi,G(\psi)\>)$ for all spinors $\varphi,\psi$. 
Thus the real structure induces a real-orthogonal and parallel decomposition 
$\Si M=\Sigma_+M \oplus \Sigma_-M$ into \emph{real} subbundles associated to the 
eigenvalues $\pm 1$. We also have $\Si M=\Si_+M\otimes_\mR\mC$.\\

Since $f_0^*TN$ is considered as a real vector bundle, 
we obtain a $\mathbb{Z}_2$-grading $G\otimes \id$
on the tensor product $\Si M\otimes f_0^*TN$ anticommuting with the
Clifford multiplication by tangent vectors.
In particular $G\otimes \id$ anti-commutes with $\Dir^{f_0}$ which hence splits into $\Dir_+^{f_0}+\Dir_-^{f_0}$, where
\[\Dir_\pm^{f_0}:C^\infty(M,\Si_\pm M\otimes f_0^*TN)\lra C^\infty(M,\Si_\mp M\otimes f_0^*TN).\]
{}From the or\-tho\-go\-na\-li\-ty of the splitting, we see that 
$\Re e(\langle\Dir^{f_0}\Phi_+,\Phi_+\rangle)=\Re e(\langle\Dir^{f_0}\Phi_-,\Phi_-\rangle)=0$ for all $\Phi_\pm\in C^\infty(M,\Si_\pm M\otimes f_0^*TN)$.
On the other hand, $(\Dir^{f_0}\Phi_\pm,\Phi_\pm)_{\mathrm{L}^2}$ is real as 
$\Dir^{f_0}$ is self-adjoint. 
Thus $(\Dir^{f_0}\Phi_+,\Phi_+)_{\mathrm{L}^2}=(\Dir^{f_0}\Phi_-,\Phi_-)_{\mathrm{L}^2}=0$.
Therefore, Corollary \ref{c:suffcondcritpointL1} implies:

\begin{corollary}\label{c:existDiracharm} 
Assume $m$ is even or $m\equiv 1\;(8)$. Let  $f_0\in C^\infty(M^m,N)$
be a harmonic map. Split
$\Phi_0\in\mathrm{ker}(\Dir^{f_0})$ into $\Phi_0=\Phi_0^++\Phi_0^-$. 
Then the pairs $(f_0,\Phi_0^+)$ and $(f_0,\Phi_0^-)$ are Dirac-harmonic.
\end{corollary}

\begin{remark}
An alternative proof of Corollary \ref{c:existDiracharm} is obtained by showing 
$V_{G\otimes\id(\Phi_0)}=-V_{\Phi_0}$ for all spinors $\varphi_0$. 
If $\Phi_0=\pm G\otimes\id(\Phi_0)$, then using $V_{-\Phi_0}=V_{\Phi_0}$, we obtain $V_{\Phi_0}=0$ and hence $(f_0,\Phi_0)$ is Dirac-harmonic.
\end{remark}


\section{The case $m\in 4\mathbb{N}$}\label{sec.m4k}

As noticed above, if we assume $m$ to be even, the Dirac operator $D^f$ is odd with respect
to the above grading.
Restriction to sections of $\Si_+ M$ thus yields
an operator
  $$\Dir^f_+:C^\infty(M,\Si_+M)\to C^\infty(M,\Si_-M).$$ 
It is a Fredholm operator and  the 
Atiyah-Singer index theorem yields its (complex) index:
   $$\ind(\Dir^f_+)= \{\mathrm{ch}(f^*TN)\cdot\wih{A}(TM)\}[M].$$
This implies 
\[\ind(\Dir^f_+)=\dim_\mC\ker(\Dir^f_+)- \dim_\mC\ker (\Dir^f_-) 
  =\left\{\begin{array}{lll} 
\al(M,f)&\textrm{ if }m\equiv 0&(8)\\ 
0 \hfill&\textrm{ if }m\equiv 2,6&(8)\\ 
2 \,\al(M,f)&\textrm{ if }m\equiv 4&(8)\\ 
\end{array}\right.\]
This follows from the definition of $\al$ if $m$ is a 
multiple of~$4$ and from the fact that the real (case $m\equiv 6\;(8)$) or quaternionic (case $m\equiv 2\;(8)$) structure on $\Si M$ exchanges $\ker \Dir^f_+$ and $\ker \Dir^f_-$ if $m\equiv 2,6\;(8)$. Alternatively, the statement in the case $m\equiv 6\;(8)$ can be deduced from 
the index theorem.

We now restrict to the case that $m$ is a 
multiple of~$4$.
Corollary~\ref{c:existDiracharm} yields the following 

\begin{corollary}\label{c:existDiracharm4k} 
Let $m\in 4\mathbb{N}$.
Assume $f_0\in C^\infty(M^m,N)$ to be a har\-mo\-nic map with $\ind(\Dir^{f_0}_+)\neq0$.
Set $\epsilon:=\mathrm{sign}(\ind(\Dir^{f_0}_+))\in \{+,-\}$.
Then $$\{(f_0,\Phi_0^\epsilon)\,|\,\Phi_0^\epsilon\in \ker(\Dir^{f_0}_\epsilon)\}$$
is a linear space of Dirac-harmonic maps of complex dimension at least $|\ind(\Dir^{f_0}_+)|$.
This complex dimension is even for $m\equiv 4\;(8)$.
\end{corollary}

\section{Minimality}\label{sec.min}  

The present and the following sections provide results for any dimension~$m$, 
but are mainly interesting 
if $m\equiv 1,2\;(8)$.

\begin{definition}
A smooth map $f_0:M\lra N$ is called \emph{perturbation-minimal} if and only if $\dim(\ker(\Dir^{f^*TN}))\geq \dim(\ker(\Dir^{f_0^*TN}))$ for all $f$ in a $C^\infty$-neighbourhood of $f_0$.
\end{definition}


Obviously, any homotopy class of maps from $M$ to $N$ contains per\-tur\-ba\-tion-mi\-ni\-mal maps.

\begin{proposition}\label{p:perturbmin}
Let $f_0:M\to N$ be perturbation-minimal and harmonic. 
Then for any $\Phi_0\in \ker(\Dir^{f_0})$, the pair $(f_0,\Phi_0)$ is 
an uncoupled Dirac-har\-mo\-nic map.
\end{proposition}

\begin{proof}
Since $f_0$ is harmonic and $\Dir^{f_0}\Phi_0=0$, it suffices by Corollary \ref{c:suffcondcritpointL1} to show that, for every smooth variation $(f_t)_{t\in]-\varepsilon,\varepsilon[}$ of $f_0$, there exists a variation $(\Phi_t)_t$ of $\Phi_0$, differentiable at $t=0$, such that $\Dir^{f_t}\Phi_t=0$ for all $t$ close enough to~$0$.
To that extent we fix such a variation $(f_t)_{t\in]-\varepsilon,\varepsilon[}$ of $f_0$ and consider, for each $t\in\,]-\varepsilon,\varepsilon[$, the L$^2$-orthogonal projection $\pi_t:\mathrm{L}^2(M,\Sigma M\otimes f_0^*(TN))\lra\ker(\wih{\Dir}_t)\subset C^\infty(M,\Sigma M\otimes f_0^*TN)$, where $\wih{\Dir}_t:=\beta_t^{-1}\circ\Dir_t\circ\beta_t$ and $\Dir_t:=\Dir^{f_t}$, see Proposition \ref{p:critpointL1}.\\
\noindent{\bf Claim A}: {\sl For every $\Psi_0\in\ker(\Dir_0)$ one has $\|\pi_t(\Psi_0)-\Psi_0\|_{\mathrm{L}^2}=\mathrm{O}(t)$ as $t\to 0$.}\\
{\it Proof of {\rm Claim A}}.
By assumption, one has $\dim(\ker(\wih{\Dir}_t))\geq \dim(\ker(\wih{\Dir}_0))=\dim(\ker(\Dir_0))$ for all $t$ small enough.
Since the dimension of the kernel is always upper semi-continuous in the parameter, this already implies the equality $\dim(\ker(\wih{\Dir}_t))=\dim(\ker(\Dir_0))=:k$ for all $t$ in a sufficiently small neighbourhood of~$0$.
Let now $\{\lambda_j(\wih{\Dir}_t)\}_{j=1}^\infty$ denote the spectrum of $\wih{\Dir}_t$, where $0\leq|\lambda_1|\leq\ldots\leq|\lambda_j|\leq|\lambda_{j+1}|\leq\ldots$.
Since the map $t\mapsto f_t$ is smooth, the spectrum of $\wih{\Dir}_t$ depends continuously on $t$ (each eigenvalue is continuous in $t$) \cite[Sec. 9.3]{Weidmann}, in particular $t\mapsto \lambda_{k+1}(\wih{\Dir}_t)$ is continuous.
But the condition $\dim(\ker(\wih{\Dir}_t))=k$ forces $\lambda_{k+1}(\wih{\Dir}_t)$ to be positive for all small enough $t$, therefore there exists $\lambda_0>0$ such that $|\lambda_{k+1}(\wih{\Dir}_t)|\geq\lambda_0$ for all small $t$.
The min-max principle then implies that, for every $\Psi\in\ker(\wih{\Dir}_t)^\perp\cap\mathrm{H}^{1,2}(M)$, 
\[\|\wih{\Dir}_t\Psi\|_{\mathrm{L}^2}\geq|\lambda_{k+1}(\wih{\Dir}_t)|\cdot\|\Psi\|_{\mathrm{L}^2}\geq\lambda_0\cdot\|\Psi\|_{\mathrm{L}^2}.\]
Putting $\Psi=\pi_t(\Psi_0)-\Psi_0\in\ker(\pi_t)\cap\mathrm{H}^{1,2}(M)=\ker(\wih{\Dir}_t)^\perp\cap\mathrm{H}^{1,2}(M)$, we obtain
\be 
\lambda_0\cdot\|\pi_t(\Psi_0)-\Psi_0\|_{\mathrm{L}^2} &\leq&\|\wih{\Dir}_t(\pi_t(\Psi_0)-\Psi_0)\|_{\mathrm{L}^2}\\
&=&\|(\wih{\Dir}_0-\wih{\Dir}_t)(\Psi_0)\|_{\mathrm{L}^2}\\
&\leq&\|\wih{\Dir}_0-\wih{\Dir}_t\|_{\mathrm{op}}\cdot\|\Psi_0\|_{\mathrm{H}^{1,2}},
\ee
where $\|\cdot\|_{\mathrm{op}}$ denotes the operator norm for bounded linear operators from $\mathrm{H}^{1,2}(M)$ into $\mathrm{L}^2(M)$.
Since by construction all operators $\wih{\Dir}_t$ have the same principal symbol and depend smoothly on the parameter $t$, it is easy to see that $\|\wih{\Dir}_0-\wih{\Dir}_t\|_{\mathrm{op}}=\mathrm{O}(t)$ and hence $\|\pi_t(\Psi_0)-\Psi_0\|_{\mathrm{L}^2}=\mathrm{O}(t)$ as $t$ goes to~$0$, which was to be shown.
\hfill$\checkmark$ 

$ $\\
\noindent{\bf Claim B}: {\sl For every $\Psi_0\in\ker(\Dir_0)$, the limit $\buil{\mathrm{lim}}{t\to 0}\frac{\pi_t(\Psi_0)-\Psi_0}{t}$ exists in $\mathrm{H}^{1,2}$ and is actually smooth on $M$.}\\
{\it Proof of {\rm Claim B}}.
For obvious reasons (see (\ref{eq:dDirdt}) above), we use the short notation $\frac{d}{dt}\wih{\Dir}_t{}_{|_{t=0}}:=\sum_{j=1}^m e_j\cdot\otimes R_{\frac{\partial f}{\partial t}(0),df_0(e_j)}^N\in C^\infty(M,\mathrm{End}(\Sigma M\otimes f_0^*TN))$.
We first show that $\frac{d}{dt}\wih{\Dir}_t{}_{|_{t=0}}\Psi_0\in\ker(\Dir_0)^\perp\cap\mathrm{H}^{1,2}(M)$ (as in Claim A, the orthogonal complement is considered in L$^2$).
Let $X_0\in\ker(\Dir_0)$, then for every small $t\neq 0$ one has $(\frac{d}{dt}\wih{\Dir}_t{}_{|_{t=0}}\Psi_0,X_0)_{\mathrm{L}^2}=\buil{\mathrm{lim}}{t\to 0}(\frac{1}{t}\wih{\Dir}_t\Psi_0,X_0)_{\mathrm{L}^2}$, where the convergence of the zero-order-operator $\frac{\Dir_t-\Dir_0}{t}$ to $\frac{d}{dt}\wih{\Dir}_t{}_{|_{t=0}}$ is to be understood in the C$^0$-sense (and not in the operator norm $\|\cdot\|_{\mathrm{op}}$ above).
Now
\be 
(\frac{1}{t}\wih{\Dir}_t\Psi_0,X_0)_{\mathrm{L}^2}&=&(\frac{1}{t}\cdot\Psi_0,\wih{\Dir}_tX_0)_{\mathrm{L}^2}\\
&=&(\Psi_0-\pi_t(\Psi_0),\frac{1}{t}\wih{\Dir}_tX_0)_{\mathrm{L}^2}
\ee
because of $\wih{\Dir}_t(\pi_t(\Psi_0))=0$.
But $\frac{1}{t}\wih{\Dir}_tX_0\buil{\lra}{t\to 0}\frac{d}{dt}\wih{\Dir}_t{}_{|_{t=0}}X_0$ (which is well-defined and lies in $C^\infty(M,\Sigma M\otimes f_0^*TN)$) and $\Psi_0-\pi_t(\Psi_0)\buil{\lra}{t\to 0}0$ in L$^2$ by Claim A, therefore $(\frac{1}{t}\wih{\Dir}_t\Psi_0,X_0)_{\mathrm{L}^2}\buil{\lra}{t\to 0}0$, and this implies   $$(\frac{d}{dt}\wih{\Dir}_t{}_{|_{t=0}}\Psi_0,X_0)_{\mathrm{L}^2}=0.$$
Since this holds for every $X_0\in\ker(\Dir_0)$, we obtain $\frac{d}{dt}\wih{\Dir}_t{}_{|_{t=0}}\Psi_0\in\ker(\Dir_0)^\perp$.
Elliptic regularity yields the smoothness of $\Psi_0$ and thus  $\frac{d}{dt}\wih{\Dir}_t{}_{|_{t=0}}\Psi_0$ is smooth as well.

Hence there exists a unique $\Theta_0\in\ker(\Dir_0)^\perp$ such that $\Dir_0\Theta_0=-\frac{d}{dt}\wih{\Dir}_t{}_{|_{t=0}}\Psi_0$.
Note that, because of the ellipticity of $\Dir_0$, the section $\Theta_0$ has to be smooth.
We show that $\buil{\mathrm{lim}}{t\to 0}\frac{\pi_t(\Psi_0)-\Psi_0}{t}=\Theta_0$ in H$^{1,2}$.
By Claim A, we have, for every $t\neq 0$,
\be
\Dir_0(\frac{\pi_t(\Psi_0)-\Psi_0}{t}-\Theta_0)&=&\frac{(\Dir_0-\Dir_t)}{t}(\pi_t(\Psi_0))-\Dir_0\Theta_0\\
&\bui{\buil{\lra}{t\to 0}}{\mathrm{L}^2}&-\frac{d}{dt}\wih{\Dir}_t{}_{|_{t=0}}\Psi_0+\frac{d}{dt}\wih{\Dir}_t{}_{|_{t=0}}\Psi_0,
\ee
that is, $\|\Dir_0(\frac{\pi_t(\Psi_0)-\Psi_0}{t}-\Theta_0)\|_{\mathrm{L}^2}\buil{\lra}{t\to 0}0$.
Since $\Dir_0$ is elliptic, it remains to show that $\|\frac{\pi_t(\Psi_0)-\Psi_0}{t}-\Theta_0\|_{\mathrm{L}^2}\buil{\lra}{t\to 0}0$.
Pick any $X_0\in\ker(\Dir_0)$, then for any $t\neq 0$,
\be 
\big(\frac{\pi_t(\Psi_0)-\Psi_0}{t}-\Theta_0,X_0\big)_{\mathrm{L}^2}&=&\big(\frac{\pi_t(\Psi_0)-\Psi_0}{t},X_0\big)_{\mathrm{L}^2}\\
&=&\big(\frac{\pi_t(\Psi_0)-\Psi_0}{t},X_0-\pi_t(X_0)\big)_{\mathrm{L}^2}.
\ee
Since by Claim A both $\frac{\pi_t(\Psi_0)-\Psi_0}{t}$ remains L$^2$-bounded near $t=0$ and $X_0-\pi_t(X_0)\buil{\lra}{t\to0}0$ in L$^2$, we deduce that $\big(\frac{\pi_t(\Psi_0)-\Psi_0}{t}-\Theta_0,X_0\big)_{\mathrm{L}^2}\buil{\lra}{t\to0}0$.
This holds for any $X_0$ in the finite-dimensional space $\ker(\Dir_0)$, therefore $\|\frac{\pi_t(\Psi_0)-\Psi_0}{t}-\Theta_0\|_{\mathrm{L}^2}\buil{\lra}{t\to0}0$, which yields the result.
\hfill$\checkmark$

It follows from Claim B that, setting $\wih{\Phi}_t:=\pi_t(\Phi_0)$, then $\wih{\Phi}_t\in C^\infty(M,\Sigma M\otimes f_0^*TN)$ is a solution of $\wih{\Dir}_t\wih{\Phi}_t=0$ for every $t$ with $\wih{\Phi}_t{}_{|_{t=0}}=\Phi_0$ and the map $]-\varepsilon,\varepsilon[\lra\mathrm{H}^{1,2}(M)$,  $t\mapsto\wih{\Phi}_t$ is differentiable at $t=0$ with $\frac{\partial\wih{\Phi}}{\partial t}(0)\in C^\infty(M,\Sigma M\otimes f_0^*TN)$.
Therefore, $\Phi_t:=\beta_t\circ\wih{\Phi}_t$ fulfills the conditions required above and the proposition is proved.
\end{proof}

%

\section{Non-minimality}\label{sec.nonmin}

Assume that a harmonic map $f_0:M\to N$ is given with
$\al(M,f_0)\neq 0$, thus in particular $\ker(\Dir^{f_0})\neq 0$.
In the previous section we have seen that, if $f_0$ is perturbation-minimal,
then for any $\Phi \in\ker(\Dir^{f_0})$ the pair $(f_0,\Phi)$ is an uncoupled Dirac-harmonic map. 
Hence we have obtained a linear space of uncoupled
Dirac-harmonic maps of real dimension $2 \dim_\mC \ker(\Dir^{f_0})$.
Although we have no proof for the moment, it seems that the perturbation-minimal case is 
the generic one.

In this section we study the case where $f_0$ is not 
perturbation-minimal.
Obviously, the proof of the perturbation-minimal setting cannot be easily adapted since the dimension of $\mathrm{ker}(\Dir_t)$ does no longer remain constant for small $t$. However, using the grading arguments 
explained in Section~\ref{sec.grad}, we will 
obtain a space of uncoupled, nontrivial 
Dirac-harmonic maps of even larger dimension. Most of the spaces --- but not all --- are actually complex vector spaces. 
Nevertheless, for homogeneity of presentation we only use real dimensions in the theorem below.

To that extent we define, for any $m\equiv 0,1,2,4\;(8)$, the integers $b_m$ and $d_m$ by the following formulas in which
the minimum runs over all compact Riemannian spin manifolds $M$ of dimension~$m$, all spin structures on $M$, 
all compact Riemannian manifolds $N$ and all
maps $f:M\to N$ with non-trivial $\al$-index.
\[ 
b_m:=\min\Big\{\dim_\R(\mathrm{ker}(\Dir^{f_0}))\,|\,f_0\in C^\infty(M^m,N),\,\al(M,f_0)\neq0\Big\}
\]
and
\be 
d_m&:=&\min\Big\{\max_{\ep\in\{\pm\}}\big\{\dim_\R(\mathrm{ker}(\Dir^{f_0}_\ep))\big\}\,|\,f_0\in C^\infty(M^m,N),\,\al(M,f_0)\neq0\\
& &\phantom{\min\Big\{}\textrm{ and }f_0\textrm{ non-perturbation-minimal}\Big\}.
\ee
Both $b_m$ and $d_m$ are positive integers because of the assumption $\al(M,f_0)\neq 0$ for all admissible $M$ and $f_0$ (and such exist in each dimension under consideration).
By Proposition~\ref{p:perturbmin}, the number $b_m$ is the minimal {\it real} dimension of the space of $\Phi_0$'s such that $(f_0,\Phi_0)$ is a Dirac-harmonic map for a given perturbation-minimal and harmonic map $f_0\in C^\infty(M^m,N)$.
In case the given map $f_0\in C^\infty(M^m,N)$ is harmonic but no more perturbation-minimal, the number $d_m$ is the minimal \emph{real} dimension of the space of $\Phi_0$'s such that $(f_0,\Phi_0)$ is Dirac-harmonic (Corollary \ref{c:existDiracharm}).
The main result of this section provides a lower bound for $d_m$.
For the sake of completeness, we include the corresponding lower bound for $b_m$.

\begin{theorem}\label{t:f0notpertmin}
With the above notations, $b_m$ and $d_m$ are bounded below by the following integers:
\begin{center}{\rm\begin{tabular}{|c||c|c|c|c|}
\hline m  mod 8& 0& 1& 2& 4\\\hline\hline 
$b_m\geq$ & 2& 2& 4& 4\\\hline 
$d_m\geq$ & 4& 3& 6& 8\\\hline 
\end{tabular}}\end{center}
In particular, $b_m\geq a_m$ and $d_m\geq a_m$ where again $a_m=2$ for $m\equiv 0,1\;(8)$ and $a_m=4$ for $m\equiv 2,4\;(8)$.
\end{theorem}

Roughly speaking, the theorem says that in the non-perturbation-minimal case the space of Dirac-harmonic maps we obtain is
larger than in the perturbation-minimal one.
Note that Theorem \ref{t:f0notpertmin} implies Theorem \ref{t:mainresult}.

\begin{proof}
We handle the four cases separately.\\
{\bf$\bullet$ Case $m\equiv 0\;(8)$}: If $\al(M,f_0)\neq0$, then $|\al(M,f_0)|\geq 1$, thus 
  $$|\dim_\mC \ker \Dir^{f_0}_+ + \dim_\mC \ker \Dir^{f_0}_-|\geq |\ind(\Dir^{f_0}_+)|\geq 1,$$ 
which implies $b_m\geq 2$.
On the other hand we have
\be 
\max_{\ep\in\{\pm\}}(\dim_\C\mathrm{ker}(\Dir^{f_0}_\ep))&=&\min_{\ep\in\{\pm\}}(\dim_\C\mathrm{ker}(\Dir^{f_0}_\ep))+|\ind(\Dir^{f_0}_+)| 
\ee
Thus, if $f_0$ is not perturbation-minimal, i.\thinspace e. if $\mathrm{ker}(\Dir^{f_0}_\ep))\neq \{0\}$, then 
  $$d_m= 2 \min \Bigl\{\buil{\max}{\ep\in\{\pm\}}(\dim_\C\mathrm{ker}(\Dir^{f_0}_\ep))\Bigr\}\geq 4.$$
%
{\bf$\bullet$ Case $m\equiv 4\;(8)$}: If $\al(M,f_0)\neq0$, then $|\al(M,f_0)|\geq 1$ as well.
Recall that, in these dimensions there exists a quaternionic structure on the (twisted) spinor bundle commuting with both the $\mathbb{Z}_2$-grading and the Dirac operator $\Dir^{f_0}$, see \eg \cite[Sec. 1.7]{friedrich:00}. 
This quaternionic structure turns the vector spaces $\Si_\pm M\otimes f_0^*TN$ into quaternionic spaces. 
Thus the discussion is a\-na\-lo\-gous to above, but all real dimensions
are divisible by $4$ instead of $2$. We obtain $b_m\geq 4$ and $d_m\geq 8$.\\
%
{\bf$\bullet$ Case $m\equiv 1\;(8)$}: If $\al(M,f_0)\neq0$, then $\al(M,f_0)=1\in\mathbb{Z}_2$, so that $\dim_\C\mathrm{ker}(\Dir^{f_0})\geq1$ in case $f_0$ is perturbation-minimal. 
Hence $b_m\geq 2$.
If $f_0$ is not perturbation-minimal, then for any $\ep\in \{+,-\}$ we have $\dim_\R\mathrm{ker}(\Dir^{f_0}_\ep)=\dim_\C\mathrm{ker}(\Dir^{f_0})\geq3$.
This shows $d_m\geq 3$.\\
{\bf$\bullet$ Case $m\equiv 2\;(8)$}: If $\al(M,f_0)\neq0$, then $\al(M,f_0)=1\in\mathbb{Z}_2$.
Recall that, in these dimensions there exists a quaternionic structure on the (twisted) spinor bundle anti-commuting with the $\mathbb{Z}_2$-grading $G\otimes\id$ and commuting with the Dirac operator $\Dir^{f_0}$, see again \eg \cite[Sec. 1.7]{friedrich:00}. Thus $\mathrm{ker}(\Dir^{f_0})$ is a quaternionic vector spaces and thus $\dim_\C\mathrm{ker}(\Dir^{f_0})\geq 2$ in the perturbation-minimal case, implying $b_m\geq 4$.
As the quaternionic structure anti-commutes with $G\otimes\id$, we also have 
$$\dim_\R\mathrm{ker}(\Dir^{f_0}_\ep)= \dim_\C\mathrm{ker}(\Dir^{f_0})$$
for $\ep=+$ and $\ep=-$.
If $f_0$ is not perturbation-minimal, then $\dim_\C\mathrm{ker}(\Dir^{f_0})\geq 6$. It follows $d_m\geq 6$.
\end{proof}

\begin{remark}\label{r:fnotpertminmequiv1mod8}
For a given Dirac-harmonic map $(f_0,\Phi_0)$ with harmonic map\-ping-component $f_0$, the pair $(f_0,\lambda\Phi_0)$ obviously remains Dirac-harmonic for every $\lambda\in\C$.
In particular the space of Dirac-harmonic maps $(f_0,\Phi_0)$ with fixed harmonic map\-ping-component $f_0$ is a complex cone.
Therefore, in dimension $m\equiv 1\;(8)$ (which is the only one where the space of $\Phi_0$'s ma\-king $(f_0,\Phi_0)$ Dirac-harmonic is \emph{a priori} only real) and for non-perturbation-minimal harmonic maps $f_0$, we actually obtain a complex cone of real dimension at least $4$ of $\Phi_0$'s such that $(f_0,\Phi_0)$ is Dirac-harmonic.
\end{remark}

\section{Examples of maps with non-trivial index}\label{sec.examples}  

In this section, we still assume that $M$ is a compact spin manifold but do not fix the spin structure since \eg in 
Corollary~\ref{c:al0nonzero} we want to choose it such that
the index does not vanish. As a consequence the spin structure $\chi$ will no longer be suppressed in $\al(M,\chi,E)$.

\subsection{The case $m=2$}\label{subsec.m=2}
In this subsection we assume $M$ to be a closed orientable connected surface. 
Such a surface always carries a spin structure~$\chi$. In general this spin structure is not unique; the space of all spin structures 
(up to isomorphism) is an affine $\mZ_2$-space, modeled on the $\mZ_2$-vector space $H^1(M,\mathbb{Z}_2)$. Thus $\chi+\al$ is a spin structure on $M$ for every $\al\in H^1(M,\mathbb{Z}_2)$.

We first compute the $\al$-genus in terms of simpler 
bordism invariants of surfaces which we use to produce new examples of Dirac-harmonic maps.

\begin{proposition}\label{p:indexDME2dim}
Let $E\lra M$ be a real vector bundle of rank $k$ over a closed connected oriented surface $M$ with spin structure $\chi$.
Let $w_1(E)\in H^1(M,\mathbb{Z}_2)$ and $w_2(E)\in H^2(M,\mathbb{Z}_2)$ be the first and second Stiefel-Whitney classes of $E\lra M$ respectively.
Then
\begin{equation}\label{eq:indthmsurf}
\al(M,\chi,E)=(k+1)\cdot\alpha(M,\chi)+\al(M,\chi+w_1(E))+w_2(E)[M].
\end{equation}
\end{proposition}

\begin{proof}
The proof consists of two steps.\\
{\bf Claim 1}: {\sl The statement holds in case that $E\lra M$ is orientable}.\\
{\it Proof of {\rm Claim 1}}.
Since any orientable real line bundle is trivial, we may assume $k\geq 2$.
First, we show that it suffices to handle the case where $M=\C\mathrm{P}^1$.
Let indeed $D^2\subset M$ be any embedded closed $2$-disc in $M$.
Then the loop $\partial D^2$ lies in the commutator subgroup $[\pi_1(M\setminus \!\bui{D}{\circ}{\!}^2),\pi_1(M\setminus \!\bui{D}{\circ}{\!}^2)]$ of $\pi_1(M\setminus \!\bui{D}{\circ}{\!}^2)$.
More precisely, $M\setminus\!\bui{D}{\circ}{\!}^2$ deformation retracts onto the wedge sum of $2g$ circles, where $g$ is the genus of $M$.
Since any orientable real vector bundle over the circle is trivial, it is also trivial on the wedge sum of circles.
This in turn implies that $E_{|_{M\setminus\!\bui{D}{\circ}{\!}^2}}\lra M\setminus\!\bui{D}{\circ}{\!}^2$ is trivial.
It follows that $M$ together with $E\lra M$ can be seen as the connected sum $M\sharp\,\C\mathrm{P}^1$, where the first factor $M$ carries the \emph{trivial} vector bundle $\underline{\R^k}:=\R^k\times M\lra M$ and $\C\mathrm{P}^1$ carries some vector bundle $F\lra\C\mathrm{P}^1$
such that $E$ is obtained by gluing $\underline{\R^k}$ together with $F$.
Here one should pay attention to the fact that, when performing the connecting sum, the bundles on both factors have to be trivialized so that their trivializations coincide on $\partial D^2$.
Now as in classical surgery theory (without bundles), any surgery between manifolds with bundles provides a spin-bordism for vector bundles.
The invariance of $\al$ under spin-bordism then gives $\al(M,\chi,E)=\al(M,\chi,\underline{\R^k})+\al(\C\mathrm{P}^1,F)$ (there is only one spin structure on $\C\mathrm{P}^1$, so we omit the notation $\chi$ in that case).
Obviously one has $\al(M,\chi,\underline{\R^k})=k\al(M,\chi,\underline{\R})=k\alpha(M,\chi)$ - in particular (\ref{eq:indthmsurf}) holds for trivial vector bundles (of any rank), since $w_2(\underline{\R^k})=0$.
Moreover, the Stiefel-Whitney number $w_2(E)[M]$ is also a spin-bordism invariant, a fact which is analogous to Pontrjagin's theorem \cite[Thm. 4.9 p.52]{MilnorStasheff} and which can be elementarily proved in just the same way.
Therefore, $w_2(E)[M]=w_2(\underline{\R^k})[M]+w_2(F)[\C\mathrm{P}^1]=w_2(F)[\C\mathrm{P}^1]$.
Hence we are reduced to showing $\al(\C\mathrm{P}^1,F)=w_2(F)[\C\mathrm{P}^1]$, which is exactly (\ref{eq:indthmsurf}) for $M=\C\mathrm{P}^1$ since~$\alpha(\C\mathrm{P}^1)=0$ (the existence of a metric with positive scalar curvature implying the vanishing of the kernel of the untwisted Dirac operator).\\
Since $\C\mathrm{P}^1$ is simply-connected, each vector bundle over $\C\mathrm{P}^1$ is orientable.
If $k\geq 3$, then there are only two isomorphism-classes of $k$-ranked real vector bundles over $\C\mathrm{P}^1$, whereas those isomorphism classes stand in one-to-one correspondence with the integral powers of the tautological complex line bundle $\gamma_1\lra\C\mathrm{P}^1$ if $k=2$, see \eg \cite[Thm. 18.5]{Steenrod51}.
Actually, it suffices to show (\ref{eq:indthmsurf}) for the tautological complex line bundle $E=\gamma_1$.
Consider indeed $l\cdot\gamma_1:=\bigoplus_{j=1}^l\gamma_1\lra\C\mathrm{P}^1$ for any $l\in\mathbb{N}$, $l\geq 1$.
Then $l\cdot\gamma_1$ is a $2l$-ranked real vector bundle over $\C\mathrm{P}^1$ and is non-trivial (its total Chern-class is $1+l\cdot a\neq 1$, where $a\in H^2(\C\mathrm{P}^1,\mathbb{Z})\cong\mathbb{Z}$ is the generator given by the tautological bundle).
Moreover, $l\cdot\gamma_1$ represents up to isomorphism the \emph{only} non-trivial $2l$-ranked real vector bundle over $\C\mathrm{P}^1$ if $l\geq 2$.
Since trivially $\al(\C\mathrm{P}^1,l\cdot\gamma_1)=l\cdot\al(\C\mathrm{P}^1,\gamma_1)$ and $w_2(l\cdot\gamma_1)=l\cdot w_2(\gamma_1)$, we are reduced to the case where $l=1$.
In case $k=2l+1$ is odd, the bundle $l\cdot\gamma_1\oplus \underline{\R}\lra\C\mathrm{P}^1$ is again $k$-ranked and non-trivial, so it is up to isomorphism the only non-trivial $k$-ranked real vector bundle over $\C\mathrm{P}^1$.
As noticed above, the spin Dirac operator on $\C\mathrm{P}^1$ has trivial kernel, so that  $\al(\C\mathrm{P}^1,l\cdot\gamma_1\oplus\underline{\R})=\al(\C\mathrm{P}^1,l\cdot\gamma_1)=l\cdot\al(\C\mathrm{P}^1,\gamma_1)$ and, as is well-known, $w_2(l\cdot\gamma_1\oplus\underline{\R})=w_2(l\cdot\gamma_1)=l\cdot w_2(\gamma_1)$, so that again we are reduced to the case where $E=\gamma_1$.
Note that the case $k=2$ can be deduced from the case $k\geq3$ by adding a trivial real line bundle: as we have seen above, $\al(M,\chi,E\oplus\underline{\R})=\al(M,\chi,E)+\al(M,\chi)$ and $w_2(E\oplus\underline{\R})=w_2(E)$.\\
It remains to show $\al(\C\mathrm{P}^1,\gamma_1)=1=w_2(\gamma_1)[\C\mathrm{P}^1]$.
On the one hand, since~$\gamma_1$ is a complex line bundle, $w_2(\gamma_1)=[c_1(\gamma_1)]_{\mathbb{Z}_2}$.
{}From $c_1(\gamma_1)[\C\mathrm{P}^1]=-1$, we obtain $w_2(\gamma_1)[\C\mathrm{P}^1]=[c_1(\gamma_1)[\C\mathrm{P}^1]]_{\mathbb{Z}_2}=1$.
On the other hand, it follows from \cite[Thm. 4.5]{GinHabibspecCPdtordu} that $\mathrm{dim}_\C(\mathrm{ker}(\Dir^{\gamma_1})=2$ (beware that we tensorize over~$\R$, so that $\Sigma T\C\mathrm{P}^1\otimes_\R\gamma_1=\{\Sigma T\C\mathrm{P}^1\otimes_\C\gamma_1\}\oplus\{\Sigma T\C\mathrm{P}^1\otimes_\C\gamma_1^{-1}\}$).
Therefore $\al(\C\mathrm{P}^1,\gamma_1)=1$, which concludes the proof of Claim 1.\hfill$\checkmark$ 

\noindent{\bf Claim 2}: {\sl The statement holds in general}.\\
{\it Proof of {\rm Claim 2}}.
By definition of the determinant bundle $\Lambda^kE\to M$, the bundle $E\oplus \Lambda^kE\to M$ is orientable.
Alternatively, this follows from $w_1(E)=w_1(\La^kE)$.
Moreover, the spinor bundle associated to the spin structure $\chi+w_1(E)$ on $M$ is $\Sigma M\otimes_\R \Lambda^kE$, in particular $\al(M,\chi,E\oplus \Lambda^kE)=\al(M,\chi,E)+\al(M,\chi,\Lambda^kE)=\al(M,\chi,E)+\al(M,\chi+w_1(E))$.
We apply the proposition in the orientable case for $E\oplus \Lambda^kE$ instead of $E$ and obtain
   $$\al(M,\chi,E)+\al(M,\chi+w_1(E))= (k+1)\al(M,\chi) + w_2(E\oplus \Lambda^kE).$$
Now we calculate $w_2(E\oplus \Lambda^kE)=w_2(E)+ w_1(E)\cup w_1(\La^k E) + w_2(\La^k E)$. As $\Lambda^kE$ is of real rank $1$ we have $w_2(\La^k E)=0$.
According to Lemma \ref{l:betacupbeta} below $w_1(E)\cup w_1(E)=0$. Thus $w_2(E\oplus \Lambda^kE)=w_2(E)$ which shows Claim 2 and concludes the proof of Proposition \ref{p:indexDME2dim}.
\end{proof}

\begin{lemma}\label{l:betacupbeta}
Let $M$ be an orientable surface and $\beta\in H^1(M,\mZ_2)$. Then $\beta\cup \beta=0$.
\end{lemma}

\begin{proof}
As $M$ is orientable, the homology group $H_k(M,\mZ)$ is a finitely generated free $\mZ$-module for $k=0,1,2$, and thus the universal coefficient
theorem implies that tensoring the coefficients with $\mZ_2$ yields isomorphisms $\mmmod 2: H^k(M,\mZ)\otimes \mZ_2\stackrel{\cong}{\lra} H^k(M,\mZ_2)$.
These isomorphisms are compatible with the $\cup$-products on $H^*(M,\mZ)$ and on $H^*(M,\mZ_2)$. Let $\overline\beta\in H^1(M,\mZ)$ such that
$\overline\beta\mmod 2=\beta$. As the cup-product on $H^*(M,\mZ)$ is skew-symmetric $\overline\beta \cup \overline\beta=0$. Thus 
  $$\beta\cup \beta = (\overline\beta\mmod 2)\cup (\overline\beta\mmod 2)= (\overline\beta \cup \overline\beta)\mmod 2 = 0.$$
\end{proof}

\textit{From now on and unless explicitly mentioned $N$ will be orientable.}

\begin{corollary}\label{c:al0nonzero}
Let $f:M\to N$ be a smooth map from a closed orientable connected surface $M$ of positive genus to an odd-dimensional o\-rien\-ta\-ble manifold $N$.
Then $M$ admits a spin structure $\chi$ such that the $\alpha$-genus $\al(M,\chi,f)\neq 0$.
In particular, each harmonic map in the homotopy class of $f$ provides a $4$-dimensional space of Dirac-harmonic maps.
\end{corollary}

\begin{proof}
For any spin structure $\chi$ on $M$ the identity (\ref{eq:indthmsurf}) reads $\al(M,\chi,f)=\alpha(M,\chi)+w_2(f^*TN)[M]$ because of $N$ being odd-dimensional and orientable.
The term $w_2(f^*TN)[M]$ does not depend on the spin structure $\chi$ on~$M$.
Since the genus of $M$ is at least one, there exists for any $x\in\mathbb{Z}_2$ at least one spin structure with $\alpha$-genus $x$.
Therefore, whatever $w_2(f^*TN)[M]$ is, $\al(M,\chi,f)$ can be made equal to $1$ for at least one spin structure on $M$.
The homotopy invariance of the $\al$-genus (Corollary \ref{c:alphahomotinv}) concludes the proof.
\end{proof}

\begin{example}\label{ex:chidependsonf}
The choice of spin structure $\chi$ for Corollary \ref{c:al0nonzero} depends on $f$ through the value $w_2(f^*TN)[M]$.
Choose for instance $E$ to be an orientable but non-spin real vector bundle of rank $3$ over $M:=\mathbb{T}^2$, \eg let $E:=(\gamma_1\sharp\,\underline{\R^2})\oplus\underline{\R}\lra\C\mathrm{P}^1\sharp\,\mathbb{T}^2=\mathbb{T}^2$ where
$\ga_1$ is the tautological bundle over $\mC P^1$ which is glued together with the trivial line bundle to give a bundle over the connected sum 
$\C\mathrm{P}^1\sharp\,\mathbb{T}^2$. Let $N$ be the total space of $E$, take $f_1:M\to N$ to be the embedding by the zero-section and consider $f_2:M\to N$ defined by $f_2(z_1,z_2):=f_1(z_1^2,z_2)$ for all $(z_1,z_2)\in\mathbb{S}^1\times\mathbb{S}^1=\mathbb{T}^2$.
Then $w_2(f_1^*TN)[M]=1$ but $w_2(f_2^*TN)[M]=0$.
\end{example}

\noindent Another straightforward consequence of Proposition \ref{p:indexDME2dim} is the following

\begin{corollary}
Let $M$ be a compact orientable connected surface with 
a bounding spin structure $\chi$, and assume that $f:M\to N$ is given.
If $N$ is orientable and $\<w_2(TN),f_*[M]\>\neq 0$ (in particular $N$ is non-spin),
then $\al(M,\chi,f)\neq 0$.
\end{corollary}

Next we apply Proposition \ref{p:indexDME2dim} and Corollary \ref{c:al0nonzero} to produce new examples of Dirac-harmonic maps.
Those examples are constructed out of harmonic maps from surfaces with spin structures such that the corresponding genus $\al(M,\chi,f)$ does not vanish.\\

Recall a result from the celebrated Sacks-Uhlenbeck paper \cite{SacksUhlenbeck_AnnMath}, proven simultaneously by L. Lemaire \cite{Lemaire78}. Given a closed orientable surface $M$ and a closed $n$-dimensional Riemannian manifold $N$, there exists in each homotopy class of maps from $M$ to $N$ an energy-minimizing (in particular harmonic) map as soon as $\pi_2(N)=0$; if $\pi_2(N)\neq 0$ then there exists a system of generators for $\pi_2(N)$ as $\pi_1(N)$-module, each containing an energy-minimizing map from $\mathbb{S}^2=\C\mathrm{P}^1$ to $N$.\\

{\bf Case $\pi_2(N)=0$}: 
We obtain $\al(\mathbb{S}^2,f)=0$ for every map $f:\mathbb{S}^2\lra N$ since $f^*TN\lra\mathbb{S}^2$ is trivial, therefore we cannot apply our main result for $M=\mathbb{S}^2$. 

If the genus of $M$ is at least $1$ and $n$ is odd, then Corollary \ref{c:al0nonzero} already shows the existence of a spin structure $\chi$ on $M$ for which $\al(M,\chi,f)\neq 0$, whatever $f:M\lra N$ is.
Thus we obtain in that case the existence of a $4$-dimensional space of Dirac-harmonic maps of the form $(f_0,\Phi)$ where
$f_0$ is a harmonic map in the homotopy class of $f$.
This enhances the previous example \cite[Thm. 2]{JostMoZhu09} (based on twistor-spinors) for either genus at least $2$ or genus $1$ and bounding spin structures on $\mathbb{T}^2$, since in both situations no non-trivial twistor-spinor is available.
If $n$ is even, then $\al(M,\chi,f)=w_2(f^*TN)[M]$.
For $n=2$ and $N$ closed one has $w_2(f^*TN)[M]=0$ since $N$ is spin, hence our main result cannot be applied.
For $n\geq 4$ the number $\al(M,\chi,f)\neq 0$ as soon as $f^*TN\lra M$ is non-trivial, since a $k(\geq 3)$-ranked real vector bundle over a surface is trivial if and only if its first and second Stiefel-Whitney classes vanish, see \eg \cite[Sec. 12]{MilnorStasheff} (in particular $N$ must be non-spin).
Note that, if $n\geq 3$ and $N$ is non-spin, then there exists at least one $M$ and one map $f$ such that $f^*TN\lra M$ is non-trivial \cite[Prop. II.1.12]{lawson.michelsohn:89}, thus providing a non-trivial Dirac-harmonic map.\\
{\bf Case $\pi_2(N)\neq 0$}:
We obtain $\al(\mathbb{S}^2,f)=w_2(f^*TN)[\mathbb{S}^2]$ for any map $f:\mathbb{S}^2\lra N$.
For $n=2$ and $N$ non-orientable (otherwise $N$ is spin) the degree of any map $f:\mathbb{S}^2\lra N$ has to be even since $\mathbb{S}^2$ is simply-connected, therefore $w_2(f^*TN)[\mathbb{S}^2]=\mathrm{deg}(f)\cdot w_2(TN)[N]=0$, so that nothing can be said. 
For $n\geq 3$ the $\al$-genus vanishes if and only if $f^*TN\lra\mathbb{S}^2$ is trivial.
As an example for non-vanishing, consider the standard embedding $\C\mathrm{P}^1\bui{\lra}{f}\C\mathrm{P}^2$, $[z_1\!:\!z_2]\lmt[z_1\!:\!z_2\!:\!0]$.
Choosing the Fubini-Study metrics of holomorphic sectional curvature $4$ on both $\C\mathrm{P}^1$ and $\C\mathrm{P}^2$, the map $f$ becomes isometric and totally geodesic, in particular harmonic.
Moreover, the pull-back bundle $f^*T\C\mathrm{P}^2$ can be identified with $\gamma_1^{-2}\oplus\gamma_1^{-1}$, in particular $w_2(f^*T\C\mathrm{P}^2)=w_2(\gamma_1)\neq 0$, so that $\al(\mathbb{S}^2,f)=1$.
Therefore there exists a non-zero $\Phi\in C^\infty(\mathbb{S}^2,\Sigma\mathbb{S}^2\otimes f^*T\C\mathrm{P}^2)$ such that $(f,\Phi)$ is Dirac-harmonic.
Actually the space of such $\Phi$'s is at least (real) $12$-dimensional and $\Phi$ may even be chosen not to come from any twistor-spinor on $\mathbb{S}^2$, see \cite{AmmGinJostMoZhu09}.
Note that, if one changes $f$ into $\wih{f}:=\iota\circ\pi$, where $\iota:\R\mathrm{P}^2\lra\C\mathrm{P}^2$ is a totally geodesic embedding and $\pi:\mathbb{S}^2\lra\R\mathrm{P}^2$ is the canonical projection, then $\wih{f}$ is harmonic but this time $\wih{f}^*T\C\mathrm{P}^2$ becomes trivial, hence no non-trivial Dirac-harmonic map can be found using our methods.\\

\noindent In an analogous way, existence results for harmonic maps by \eg Eells-Sampson \cite{EellsSampson} (see also \cite{SchoenYau76}) or Y.L. Xin \cite{Xin84} give Dirac-harmonic maps provided the cor\-res\-pon\-ding $\al$-genus does not vanish.

We summarize the results in the following two theorems, where we assume all surfaces to be connected.
Note that if we say that there is a $4$-dimensional space of Dirac-harmonic maps, we do not exclude the possibility that there is even a space of higher dimension.
We distinguish both cases where the homotopy class of the underlying map can be prescribed (Theorem \ref{t:examplesDiracharmdim2}) from the one where it cannot (Theorem \ref{t:examplesDiracharmdim2pi2neq0}).

\begin{theorem}\label{t:examplesDiracharmdim2}
Let $M$ be a closed orientable surface of genus $g$ and $N$ be a closed $n(\geq 3)$-dimensional orientable Riemannian manifold. 
Let $f:M\lra N$ be an arbitrary smooth map.
\beit
\item[$i)$] Assume $\pi_2(N)=0$, $g> 0$ and $n$ odd.
Then there exist a harmonic map $f_0$ homotopic to $f$ and at least $3$ spin structures $\chi$ on $M$ such that 
there is a $4$-dimensional space of Dirac-harmonic maps of the form $(f_0,\Phi)$. 
\item[$ii)$] Assume $\pi_2(N)=0$, $g>0$, $n$ even, and $f^*TN\lra M$ non-trivial.
Then there is a harmonic map $f_0$ homotopic to $f$ such that for any spin structure on~$N$
there is a $4$-dimensional space of Dirac-harmonic maps of the form $(f_0,\Phi)$. 
\item[$iii)$] Assume $N$ has non-positive sectional curvature or no focal point.
Let $f:M\lra N$ be any smooth map.
If $n$ is odd and $g\geq 1$, then there exists a spin structure $\chi$ on $M$ for which a non-trivial Dirac-harmonic map $f_0$ exists with $f_0$ homotopic to $f$.
If $n$ is even or $g=0$, then the same holds as soon as $f^*TN\lra M$ is non-trivial.
\eeit
\end{theorem}

In the case $\pi_2(N)\neq 0$ and if $N$ does not satisfy the conditions of iii), bubbling-off can happen, and in general one cannot prescribe
a homotopy class $[f]\in [M,N]$. It only can be prescribed modulo $\pi_2(N)$. However similar statements follow in this case, summarized in 
Theorem~\ref{t:examplesDiracharmdim2pi2neq0}. 

In this theorem two modifications of the spin condition appear that shall be discussed now.

The second Stiefel-Whitney class of $TN$, denoted by $w_2(TN)\in H^2(N,\mZ_2)=\Hom(H_2(N,\mZ_2),\mZ_2)$, can be composed
with the map $\mmmod 2: H_2(N,\mZ)\to H_2(N,\mZ_2)$ and then with the Hurewicz map $h:\pi_2(N)\to H_2(N,\mZ)$.
Then $w_2(TN)=0$ if and only if $N$ is spin, and it is an elementary exercise to show that $w_2(TN)\circ \mmod 2 \circ h=0$ 
if and only if the universal covering $\witi N$ of $N$ is spin.

If $w_2(TN)\circ \mmod 2$ vanishes, then $w_2(TN)$ factors over the Bockstein map $\beta:H_2(N,\mZ_2)\to H_1(N,\mZ)$, i.e.\ there is a homomorphism
$w: \mathrm{im}(\beta)\to \mZ_2$ with $w_2(TN)=w\circ \beta$.
Choose a complement $\Gamma$ of the $2$-torsion subgroup of $H_1(N,\mZ)$, and let 
$\wihat \Gamma\subset \pi_1(N)$ be its preimage under the Hurewicz map $\pi_1(N)\to H_1(N,\mZ)$. Let $\wihat N\to N$ be the covering associated
to $\wihat \Gamma$. This is a finite covering and the number of sheets is the order of the $2$-torsion group of $H_1(N,\mZ)$, thus a power of~$2$.
Since $\mathrm{im}(\beta)$ is included in the $2$-torsion subgroup of $H_1(N,\mZ)$, the composition $H_2(\wihat N,\mZ)\to H_2(N,\mZ)\stackrel{w_2(TN)}{\to} \mZ_2$ vanishes, and it follows that $\wihat N$ is spin.

This provides a necessary condition for the vanishing of $w_2(TN)\circ \mmod 2$, but it is unclear to us whether this is sufficient as well.

\begin{theorem}\label{t:examplesDiracharmdim2pi2neq0}
Let $N$ be a closed $n(\geq 3)$-dimensional orientable Rie\-man\-nian manifold. 
\beit
\item[$iv)$] Assume that the universal covering $\witi N$ of $N$ is not spin.
Then there is a non-constant harmonic map $f_0:\mathbb{S}^2\to N$ such that $f_0^*TN$ is non-trivial. 
In particular, there is a $4$-dimensional space of Dirac-harmonic maps of the form $(f_0,\Phi)$.
\item[$v)$] Assume that $\witi N$ is spin, even-dimensional and that $w_2(TN)$ defines a non-trivial map $H_2(N,\mZ)\to \mZ_2$.
Then there is an oriented closed surface $M$ of genus $g\geq 1$ and a non-constant harmonic map $f_0:M\to N$ such that for any spin structure on $M$
there is a $4$-di\-men\-sio\-nal space of Dirac-harmonic maps of the form $(f_0,\Phi)$.
\item[$vi)$] Assume that $\witi N$ is spin, odd-dimensional and that $w_2(TN)$ defines a non-trivial map $H_2(N,\mZ)\to \mZ_2$.
Then there is an oriented closed surface $M$ of genus $g\geq 1$, a non-constant harmonic map $f_0:M\to N$, and at least $3$ spin structures
on $M$ such that there is a $4$-dimensional space of Dirac-harmonic maps of the form $(f_0,\Phi)$. 
\item[$vii)$] Assume that $w_2(TN)$ defines a trivial map $H_2(N,\mZ)\to \mZ_2$ (in particular $\witi N$ is spin) and that $n$ is odd.  
Then for every oriented closed surface $M$ of genus $g\geq 1$ and for every smooth map $f:M\to N$ with $f_*[M]\notin h(\pi_2(N))$, there exists a non-constant harmonic map $f_0:M\to N$ and at least one spin structure
on $M$ such that there is a $4$-dimensional space of Dirac-harmonic maps of the form $(f_0,\Phi)$. 
\item[$viii)$] Assume that $w_2(TN)$ defines a trivial map $H_2(N,\mZ)\to \mZ_2$, and that $N$ is not spin.
Then there is a non-orientable closed surface $\ol M$ on which
there is a non-constant harmonic map $\bar f_0:\ol M\to N$ with 
$\<\ol f_0^*w_2(TN),[\ol M]\>\neq 0$. 
\begin{itemize}
\item[$viii-\al)$]
Assume that $\ol M$ is not diffeomorphic to $\mR \mathrm{P}^2$ and that $n$ is odd.
The orientation covering $\wihat M\to \ol M$ defines a non-constant harmonic map $\hat f_0:\wihat M\to N$. 
Then there is at least one spin structure on~$\wihat M$ for which
there is a $4$-dimensional space of Dirac-harmonic maps of the form $(\hat f_0,\Phi)$.
\item[$viii-\beta)$]
Assume $\ol M=\mR \mathrm{P}^2$.
Then the induced map $\pi_1(\mR \mathrm{P}^2)\to \pi_1(N)$ is injective.
Thus $\pi_1(N)$ has an element of order $2$.
\end{itemize}
\eeit
\end{theorem}

In $viii)$ we used $[\ol M]$ for the $\mZ_2$-fundamental class of $\ol M$.
Note that, in the cases where either $H_2(N,\mZ)\bui{\to}{w_2(TN)} \mZ_2$ vanishes and $n$ is even or $\ol M=\mR \mathrm{P}^2$ (case $viii-\beta)$) we cannot deduce the existence of non-trivial Dirac-harmonic maps.

\begin{proof} We prove all cases separately.\\
$iv)$ Choose any smooth map $f:\mathbb{S}^2\to N$ with $\<f^*w_2(TN),[\mathbb{S}^2]\>\neq 0$.
Applying Sacks-Uhlenbeck's method (Theorem \ref{theo.su.bord.or}), one obtains a harmonic map $f_0:\coprod_{j=1}^l\mathbb{S}^2\to N$ which is spin bordant to $f$, in particular $\<f^*w_2(TN),[\mathbb{S}^2]\>=\<f_0^*w_2(TN),[\coprod_{j=1}^l\mathbb{S}^2]\>=\sum_{j=1}^l\<f_{0j}^*w_2(TN),[\mathbb{S}^2]\>$, where $f_{0j}:\mathbb{S}^2\to N$ denotes the (harmonic) map induced by $f_0$ on the $j^{\textrm{th}}$ connected component of $\coprod_{j=1}^l\mathbb{S}^2$.
In particular there is at least one $j\in\{1,\ldots,l\}$ with $\<f_{0j}^*w_2(TN),[\mathbb{S}^2]\>\neq 0$, which already implies that $f_{0j}$ is non-constant.
Therefore, the harmonic map $f_{0j}:\mathbb{S}^2\to N$ satisfies $\al(\mathbb{S}^2,f_{0j})=\<f_{0j}^*w_2(TN),[\mathbb{S}^2]\>\neq 0$ and hence provides a $4$-dimensional space of Dirac-harmonic maps.\\
$v)$ Since $H_2(N,\mZ)\bui{\to}{w_2(TN)} \mZ_2$ does not vanish, there exists a closed orientable surface $M$ and a smooth map $f:M\to N$ such that $\<f^*w_2(TN),[M]\>\neq 0$.
Necessarily $M$ is of positive genus because of $\witi N$ being spin.
Theorem \ref{theo.su.bord.or} yields a harmonic map $f_0:M\amalg\coprod_{j=1}^l\mathbb{S}^2\to N$ which is bordant to $f$, in particular $\<f^*w_2(TN),[M]\>=
\<f_0^*w_2(TN),[M\amalg\coprod_{j=1}^l\mathbb{S}^2]\>=\<f_{00}^*w_2(TN),[M]\>+\sum_{j=1}^l\<f_{0j}^*w_2(TN),[\mathbb{S}^2]\>$ , where $f_{00}:M\to N$ and $f_{0j}:\mathbb{S}^2\to N$ denote the maps induced on each connected component.
Again, since $\witi N$ is spin, each $\<f_{0j}^*w_2(TN),[\mathbb{S}^2]\>$ vanish, so that $\<f_{00}^*w_2(TN),[M]\>\neq 0$, in particular $f_{00}$ is a non-trivial harmonic map.
Because of $n$ being even, we obtain $\al(M,\chi,f_{00})=\<f_{00}^*w_2(TN),[M]\>\neq 0$ for every spin structure $\chi$ on $M$, which shows the result in that case.\\
$vi)$ The only difference with $v)$ lies in the formula $\al(M,\chi,f_{00})=\al(M,\chi)+\<f_{00}^*w_2(TN),[M]\>$.
Choosing those spin structures $\chi$ with $\al(M,\chi)=0$ (and there are at least $3$ of those) one obtains the statement.\\
$vii)$ For any closed orientable surface $M$ of genus $g\geq 1$ and any smooth map $f:M\to N$ with $f_*[M]\notin h(\pi_2(N))$, Theorem \ref{theo.su.bord.or} yields a harmonic map $f_0:M\to N$ which has to satisfy $\<f_0^*w_2(TN),[M]\>=0$ because of the vanishing of the map $H_2(N,\mZ)\to \mZ_2$.
However, because only $2$-spheres have possibly popped off $M$, the map $f_0$ also satisfies $f_0{}*[M]\notin h(\pi_2(N))$, in particular $f_0$ cannot be constant. 
Since $n$ is odd and $g\geq 1$, there exists at least one spin structure $\chi$ on $M$ with $\al(M,\chi)=1$, in particular $\al(M,\chi,f_0)=\al(M,\chi)=1$ and the statement follows.\\
$viii)$ In case $H_2(N,\mZ)\bui{\to}{w_2(TN)} \mZ_2$ vanishes but $N$ itself is non-spin, there exists a (necessarily) non-orientable closed surface $\ol M$ and a smooth map $\ol f:\ol M\to N$ with $\<\ol f^*w_2(TN),[\ol M]\>\neq 0$.
Theorem \ref{theo.su.bord.non-or} gives a harmonic map $\ol f_0:\ol M\to N$ with $\<\ol f_0^*w_2(TN),[\ol M]\>=\<\ol f^*w_2(TN),[\ol M]\>\neq 0$ (recall that none of the bubbles contribute to the second Stiefel-Whitney number because of $\witi N$ being spin).
In particular $\ol f_0$ cannot be constant.
Obviously, the induced map $\hat f_0$ on $\wihat M$ remains harmonic and non-constant.
Beware however that, since $\wihat M\to \ol M$ is two-fold, $\<\hat f_0^*w_2(TN),[\wihat M]\>=2\cdot\<\ol f_0^*w_2(TN),[\ol M]\>=0$, so that $\al(\wihat M,\chi,\hat f_0)=n\cdot\al(\wihat M,\chi)$ for any spin structure $\chi$ on $\wihat M$.
In case $n$ is odd and $\wihat M$ is not diffeomorphic to $\mathbb{S}^2$ (that is, $\ol M$ is not diffeomorphic to $\mR \mathrm{P}^2$), there exists at least one spin structure $\chi$ on $\wihat M$ with $\al(\wihat M,\chi)=1$, hence $\al(\wihat M,\chi,\hat f_0)=1$, which proves $viii-\al)$.
In case $\ol M=\mR \mathrm{P}^2$, assume the group homomorphim $\pi(\mR \mathrm{P}^2)\to\pi_1(N)$ induced by $\ol f_0$ were trivial. 
Then $\ol f_0$ could lift to a smooth map $\tilde{f}_0:\mR \mathrm{P}^2\to\witi N$ through the universal cover $\witi N\to N$.
Since by assumption $\witi N$ is spin one would have $0=\<\tilde{f}_0^*w_2(T\witi N),[\mR \mathrm{P}^2]\>=\<\ol f_0^*w_2(TN),[\mR \mathrm{P}^2]\>$, contradiction.
Therefore $\pi(\mR \mathrm{P}^2)\to\pi_1(N)$ is injective.
This shows $viii-\beta)$ and concludes the proof.
\end{proof}


\subsection{The case $m\geq 3$}\label{subsec.mge3}

In higher dimensions, existence results for harmonic maps still provide non-trivial examples of Dirac-harmonic maps.
Consider for instance the situation where the closed target manifold $N$ has non-positive sectional curvature.
Then there exists an energy-minimizing (hence harmonic) map in every homotopy class of smooth maps from any closed manifold $M$ into $N$ \cite{EellsSampson,SchoenYau76}. 
As an application, pick \emph{any} closed connected Riemannian spin manifold $N$ with non-positive sectional curvature and dimension $n\equiv1\;(8)$.
Let $N'$ be \emph{any} $n$-dimensional closed Riemannian spin manifold with non-vanishing $\alpha$-genus.
In dimension $n=m=9$ the Riemannian pro\-duct of a Bott manifold with an $\mathbb{S}^1$ with non-bounding spin structure gives an example of such a manifold.
Define the map $f:N\amalg N'\to N$, $f_{|_N}:=\id_N$ and $f_{|_{N'}}:={\rm cst}$.
Obviously, the restricted map $f_{|_{N'}}$ being constant, the vector bundle $f_{|_{N'}}^*TN$ is trivial and hence 
\be \alpha(N\amalg N',f)&=&\alpha(N,f_{|_N})+\alpha(N',f_{|_{N'}})\\
&=&\alpha(N,\id_N)+n\alpha(N')\\
&=&\alpha(N,\id_N)+\alpha(N')\\
&=&\alpha(N,\id_N)+1.
\ee
On the other hand, since disks are contractible and $N$ is connected, the map $f$ induces a (smooth) map $f^\sharp:M:=N\sharp N'\to N$ which is spin bordant to $f$.
By the spin-bordism-invariance of the $\alpha$-genus, $\alpha(M,f^\sharp)=\alpha(N,\id_N)+1$.
Therefore, either $\alpha(N,\id_N)=1$ or $\alpha(M,f^\sharp)=1$.
A $2$-dimensional space of Dirac-harmonic maps is provided by $(N,\id_N)$ in the first case and by $(M,f^\sharp)$ in the second case.

\appendix

\section{Bubbling-off for harmonic maps}\label{sec.bubble}

In their celebrated article \cite{SacksUhlenbeck_AnnMath} Sacks and Uhlenbeck showed the existence of harmonic maps under certain conditions.
In particular, they explained that bubbling-off effects play an important role. However, the conclusions one may draw out of the proofs
of their article are stronger than the statements of their theorem. In particular, certain relations to bordisms hold, see 
Theorems~\ref{theo.su.bord.a} to~\ref{theo.su.bord.or}. As these theorems are used in our article, we will explain them and give other references 
that finally yield complete proofs.
We also need harmonic maps defined on non-orientable surfaces, and modifications of the results of \cite{SacksUhlenbeck_AnnMath}
also hold in this non-orientable case.
In order to fa\-ci\-li\-ta\-te the comparison to \cite{SacksUhlenbeck_AnnMath} 
we adapt to their notation to a large extend. One considers maps $s:M\to N$ from a closed Riemannian surface $M$ with metric $g$ 
to a closed Riemannian manifold $N$ with metric $h$.

\subsection{Non-orientability of $M$}
As already said above, in  \cite{SacksUhlenbeck_AnnMath} it is claimed that $M$ should be orientable. This assumption can be easily removed.
Orientability of~$M$ is needed to define
the quadratic differential $\phi$ in \cite{SacksUhlenbeck_AnnMath} and to prove \cite[Lemma~1.5]{SacksUhlenbeck_AnnMath} which states $\pa \phi=0$ for 
harmonic maps.
One easily verifies that the only parts in \cite{SacksUhlenbeck_AnnMath} using orientation are parts of the introduction, the definition of $\phi$, 
Lemma~1.5 and Theorem~1.6. 

As a consequence, in all existence results of \cite{SacksUhlenbeck_AnnMath}  the ``orientability of~$M$''-assumption can be removed. 
Theorem 1.6 provides an interpretation for harmonic maps with the special property $\phi=0$. By slightly exchanging notation, this also works for 
non-orientable surfaces, as we will explain now.
On Riemann surfaces quadratic differentials are in natural bijection to tracefree symmetric $2$-tensors, and $\pa$ then turns into the 
divergence. Thus the condition $\pa\phi=0$ is equivalent to saying that the $g$-trace-free part of $s^*h$ has vanishing divergence. This is usually
written as $\divv (s^*h)_0=0$. Lemma~1.5 then says: If $s$ is harmonic, then $(s^*h)_0$ is divergence free. Theorem~1.6 then states: If $s$ is harmonic and if $s^*h_0=0$, then $s$ is a (conformal) branched immersion.
Since those conditions are independent on the orientability of~$M$, both proofs also work in the non-orientable case.

\subsection{Bubbling-off and bordisms}

We recall further notations from \cite{SacksUhlenbeck_AnnMath}.
For $\alpha\geq 1$ they define the energy functional $E_\al(s):=\int_M (1+|ds|^2)^\al\;dv^M$.
The map $s_0:M\to N$ is harmonic if and only if it is a stationary point of the energy functional $E_1$. One can try to fix a homotopy class $[M,N]$ of maps $M\to N$ 
and to minimize $E_1$ in this class. Unfortunately the direct method in the calculus of variation fails. 
The reason for this is that $E_1$ is conformally invariant in the following sense:
if $g_1$ and $g_2$ are 
conformal metrics with $\vol(M,g_1)=\vol(M,g_2)$ then the functional $E_1$ defined with respect to $g_1$ coincides with the functional $E_1$ defined for $g_2$.
The problem is avoided if one minimizes 
$E_\al$ for $\al>1$ instead. For $\al>1$ Sacks and Uhlenbeck showed that any homotopy class $[M,N]$ contains a smooth map $s_\al$ minimizing 
$E_\al$. It remains to discuss whether and in which sense the functions $s_\al$ converge to a harmonic map $s:M\to N$ for $\al\to 1$.
An interesting bubbling-off phenomenon appears in \cite{SacksUhlenbeck_AnnMath}.  
As one often passes to subsequences, we will take a sequence  $\al_i>1$, $i\in \mN$, converging to $1$, and substitute $s_i:=s_{\al_i}$.

The following theorem can be deduced from the proofs in \cite{SacksUhlenbeck_AnnMath}:
\begin{theorem}\label{theo.su.bord.a}
Assume that $M$ is a closed surface, and that $n=\dim N\geq 3$. Further let
$s_i:M\to N$ be stationary points of $E_{\al_i}$ where $\al_i>1$, $i\in \mN$, is a sequence 
converging to $1$, and assume that  $E_{\al_i}(s_i)$ converges. Then after passing to a subsequence of indices $i$ there is 
a harmonic map $s:M\amalg \coprod_{j=1}^\ell \mathbb{S}^2\to N$ for some $\ell\in \mN\cup\{0\}$ 
such that $s_i$ converges to $s$ in a certain sense. 
\end{theorem}

The convergence of $s_i$ to $s$ is as follows: there are finitely many poins $x_1$,\ldots, $x_l$ such that  $s_i$ converges
to $s$ in $C^1(U\setminus \{x_1,\ldots,x_l\})$, see  \cite[Theorem 4.4]{SacksUhlenbeck_AnnMath}. In $x_j$ spheres may bubble-off. We work in
isothermic coordinates around $x_j$, i.e.~$x_i\cong 0$ is the base point, and the coordinates are defined on the neighborhood $U(x_i)\cong B_{2r}(0)$. 
Then there are sequences $y_i\to 0$ and $\ep_i\to 0$ such that 
$w\mapsto s_i(\ep_i w+ y_i)$, defined on a ball of radius $r\ep_i^{-1}$, converges to a map $\ti s:\mR^2\to N$ in the $C^1$-topology. This map $\ti s$ is harmonic. Conformally
identifying $\mR^2$ with $\mathbb{S}^2\setminus\{p\}$ one obtains a harmonic map  $\mathbb{S}^2\setminus\{p\}\to N$, which can be extended to a harmonic map $\mathbb{S}^2\to N$, see
 \cite[Theorem 4.6]{SacksUhlenbeck_AnnMath}. This is the restriction of the map $s$ above to one of the spheres $\mathbb{S}^2$.
The situation is even slightly more involved as spheres can bubble-off simultaneously on different growth scales in the same point $x_j$.
In other words it might happen that on a given bubble, a finite numbers of new bubbles emerge. 
And this bubbles-on-bubbles-phenomenon can be iterated to bubbles on bubbles on bubbles, etc..
However, for a given manifold $N$ and a given upper bound on the energy 
(which trivially exists as we consider minimizing maps, see also \cite[Prop. 2.4]{SacksUhlenbeck_AnnMath}),
a finite number of iterations suffices. This leads to the notion of bubble tree.
The precise way in which the maps $s_\al$ converge to such a bubble tree was analyzed by Parker in \cite{parker:96}. 
We also recommend \cite{parker:03} for a more informal introduction to bubble trees.

This precise description of the bubbling-off procedure also implies the following theorem.

\begin{theorem}\label{theo.su.bord.non-or}
With the assumptions of the previous theorem, the subsequence can be chosen such that 
there is a $3$-dimensional manifold $W$ with boundary 
$\pa W= M\amalg \coprod_{j=1}^\ell \mathbb S^2\amalg M$ with smooth maps $F_i,F:W\to N$
extending~$s_i$ and~$s$.
\end{theorem}

We will now describe the manifold $W$.
We assume for simplicity that $M$ is connected, otherwise one should consider each connected component se\-pa\-ra\-te\-ly.
The manifold $W$ is obtained by attaching $\ell$ copies of a $1$-handle to $W_0:=(M\amalg \coprod_{j=1}^\ell \mathbb{S}^2)\times[0,1]$. Such an attachment
is carried out inductively as follows, see e.g.\ \cite[VI.6]{kosinski:93} for details. We combine two maps $\overline{B_1(0)}\to M$ and  $\overline{B_1(0)}\to S^2$ which are diffeomorphisms onto balls in $M$ and $\mathbb{S}^2$, to 
a map $G:\overline{B_1(0)}\amalg \overline{B_1(0)}\to (M\amalg \mathbb{S}^2)\times \{1\}\subset \pa W_0$. We define $W_1:= W_0 \cup_G ([-1,1]\times  \overline{B_1(0)})$,
where the subscript $G$ indicates that we identify each point in $\{-1,1\}\times \overline{B_1(0)})= \overline{B_1(0)}\amalg \overline{B_1(0)}$ with its image
under $G$ in $\pa W_0$. Strictly speaking $W_1$ is a smooth manifold with boundary and a corner, but this corner can be smoothed out by slightly 
changing the differentiable structure close to the corner. Then $W_1$ is a manifold with boundary, the boundary consists of two pieces: one piece is 
$(M\amalg \coprod_{j=1}^\ell \mathbb{S}^2)\times\{0\}\subset W_0\subset W_1$. The remaining piece is diffeomorphic to $M\amalg \coprod_{j=1}^{\ell-1}\mathbb{S}^2$.
Several choices are done here, but they do not alter the diffeomorphism type of $W_1$.

By attaching a further $1$-handle we obtain a manifold $W_2$ with boundary 
$(M\amalg \coprod_{j=1}^\ell \mathbb{S}^2) \amalg (M\amalg \coprod_{j=1}^{\ell-2}\mathbb{S}^2)$, and finally after $\ell$ such attachments, we get  $W=W_\ell$ with boundary 
$(M\amalg \coprod_{j=1}^\ell \mathbb{S}^2) \amalg M$.

If $M$ is orientable, we even get more structure on $W$. In this case $M$ also admits a spin structure, and this defines an orientation and spin structure
on $M\times [0,1]$. Orientations can then be chosen on the spheres $\mathbb{S}^2$ such that the induced orientation on $W_0$ extends to $W$. 
Similarly the spin structure on $M$ yields a unique spin structure on $W$. The manifold $W$ then has the boundary 
  $$\pa W=-(M\amalg \coprod_{j=1}^\ell \mathbb{S}^2) \amalg M.$$
Here the minus sign indicates that this piece of the boundary carries the opposite orientation,
and all pieces of the boundary carry the spin structures induced from $W$. Such a $W$ is called a spin-bordism
from $M\amalg \coprod_{j=1}^\ell \mathbb{S}^2$ to $M$.

For orientable surfaces we have thus strengthened the previous theorem:

\begin{theorem}\label{theo.su.bord.or}
With the hypotheses of the previous theorems, assume furthermore that $M$ carries a fixed orientation and spin structure.
Then the $3$-dimensional manifold $W$ carries an orientation and spin structure, such that it is a spin bordism from 
$M\amalg \coprod_{j=1}^\ell \mathbb{S}^2$ to $M$. 
\end{theorem}

\subsection*{Acknowledgements}
The authors want to thank Guofang Wang, J\"ur\-gen Jost, Ulrich Bunke, Frank Pf\"affle, and Helmut Abels for interesting discussions and interesting talks related to this article.

\bibliographystyle{amsplain}

\begin{thebibliography}{10}

\bibitem{ammann.dahl.humbert:p09}
B.~Ammann, M.~Dahl, and E.~Humbert, \emph{Harmonic spinors and local
  deformations of the metric}, Preprint, {A}r{X}iv 0903.4544, to appear in
  Math. Res. Letters, 2009.

\bibitem{ammann.dahl.humbert:09}
\bysame, \emph{Surgery and harmonic spinors}, Adv. Math. \textbf{220} (2009),
  523--539.

\bibitem{AmmGinJostMoZhu09}
B. Ammann and N. Ginoux, \emph{Examples of Dirac-harmonic maps after Jost-Mo-Zhu}, 
a\-vai\-la\-ble at \texttt{http://www.mathematik.uni-regensburg.de/ginoux} (2011).

\bibitem{baer:96b}
C.~B{\"a}r, \emph{Metrics with harmonic spinors}, Geom. Funct. Anal. \textbf{6}
  (1996), 899--942.

\bibitem{baer:97b}
C.~B{\"a}r, \emph{Harmonic spinors for twisted {D}irac operators}, Math.
  Ann. \textbf{309} (1997), no.~2, 225--246.

\bibitem{baer.dahl:02}
C.~B{\"a}r and M.~Dahl, \emph{{Surgery and the Spectrum of the {D}irac
  Operator}}, J. reine angew. Math. \textbf{552} (2002), 53--76.

\bibitem{baer.schmutz:92}
C.~B{\"a}r and P.~Schmutz, \emph{{Harmonic spinors on Riemann surfaces.}}, Ann.
  Global Anal. Geom. \textbf{10} (1992), 263--273.

\bibitem{ChenJostLiWang05}
Q. Chen, J. Jost, J. Li and G. Wang, \emph{Regularity theorems and energy identities for Dirac-harmonic maps}, Math. Z. \textbf{251} (2005), no. 1, 61--84.

\bibitem{ChenJostLiWang06}
\bysame, \emph{Dirac-harmonic maps},  Math. Z. \textbf{254} (2006), no. 2, 409--432. 

\bibitem{ChenJostWang07}
Q.~Chen, J.~Jost and G.~Wang, \emph{Liouville theorems for Dirac-harmonic maps}, J. Math. Phys. \textbf{48} (2007), no. 11, 113517, 13 pp.

\bibitem{chen.jost.wang:p07a}
\bysame, \emph{Nonlinear {D}irac equations on {R}iemann
  surfaces},  Ann. Global Anal. Geom. \textbf{33} (2008), no. 3, 253--270.

\bibitem{dahl:08}
M.~Dahl, \emph{On the space of metrics with invertible {D}irac operator},
  Comment. Math. Helv. \textbf{83} (2008), no.~2, 451--469.

\bibitem{deligne.freed:99}
P.~Deligne, D.~Freed, \emph{Supersolutions}, pp.~227--356. 
In P.~Deligne, P.~Etingof, D.S. Freed, L.C. Jeffrey, D.~Kazhdan,
J.W. Morgan, D.R. Morrison, and E.~Witten (eds.), \emph{Quantum
fields and strings: a course for mathematicians. {V}ol. 1}, Providence, RI,
AMS, 1999.

\bibitem{EellsSampson}
J.~Eells and J.H.~Sampson, \emph{Harmonic mappings of Riemannian manifolds}, Amer. J. Math. \textbf{86} (1964), 109--160.

\bibitem{Freed5lect}
D.~Freed, \emph{Five lectures on supersymmetry}, American Mathematical Society, Providence, RI, 1999.

\bibitem{friedrich:00}
T.~Friedrich, \emph{Dirac operators in Riemannian geometry}, Graduate Studies in Mathematics \textbf{25}, American Mathematical Society, Providence, RI, 2000.

\bibitem{GinHabibspecCPdtordu}
N.~Ginoux and G.~Habib, \emph{The spectrum of the twisted Dirac operator on K\"ahler sub\-ma\-ni\-folds of the complex projective space}, to appear in manuscripta math.

\bibitem{hijazi:01}
O.~Hijazi, \emph{{Spectral properties of the {D}irac operator and geometrical structures}}, in: Geometric methods for quantum field theory (Villa de Leyva, 1999),  116--169, World Sci. Publ., River Edge, NJ, 2001.

\bibitem{hitchin:74}
N.~Hitchin, \emph{Harmonic spinors}, Adv. Math. \textbf{14} (1974), 1--55.

\bibitem{isobe:11}
T.~Isobe, \emph{Nonlinear {D}irac equations with critical nonlinearities on
  compact {S}pin manifolds}, J. Funct. Anal. \textbf{260} (2011), no.~1,
  253--307.

\bibitem{JostMoZhu09}
J.~Jost, X.~Mo and M.~Zhu, \emph{Some explicit constructions of Dirac-harmonic maps}, J. Geom. Phys.  \textbf{59} (2009), no. 11, 1512--1527.

\bibitem{kosinski:93}
A.~A. Kosinski, \emph{Differential manifolds}, Pure and Applied Mathematics,
  vol. 138, Academic Press Inc., Boston, MA, 1993.

\bibitem{lawson.michelsohn:89}
H.B.~Lawson and M.-L.~Michelsohn, \emph{Spin geometry}, Princeton Mathematical Series \textbf{38}, Princeton University Press, 1989.

\bibitem{Lemaire78}
L.~Lemaire, \emph{Applications harmoniques de surfaces riemanniennes}, J. Diff. Geom. \textbf{13} (1978), no. 1, 51--78. 

\bibitem{MilnorStasheff}
J.W.~Milnor and J.D.~Stasheff, \emph{Characteristic classes}, Annals of Mathematics Studies \textbf{76}, Princeton University Press, 1974.

\bibitem{Mo10}
X.~Mo, \emph{Some rigidity results for Dirac-harmonic maps}, Publ. Math. Debrecen \textbf{77} (2010), no. 3-4, 427--442.

\bibitem{parker:96}
T.H.~Parker, \emph{Bubble tree convergence for harmonic maps}, 
J. Diff. Geom. \textbf{44} (1996), 595--633.

\bibitem{parker:03}
\bysame, \emph{What is a bubble tree?}, Notices AMS \textbf{50} (2003),
  666--667.

\bibitem{SacksUhlenbeck_AnnMath}
J.~Sacks and K.~Uhlenbeck, \emph{The existence of minimal immersions of $2$-spheres}, Ann. of Math. (2)  \textbf{113}  (1981), no. 1, 1--24.

\bibitem{SchoenYau76}
R.~Schoen and S.-T.~Yau, \emph{Harmonic maps and the topology of stable hypersurfaces and manifolds with non-negative Ricci curvature}, Comment. Math. Helv.  \textbf{51} (1976), no. 3, 333--341.

\bibitem{seegerdiss}
L.~Seeger, \emph{Harmonische {S}pinoren auf gerade-dimensionalen
  {S}ph{\"a}ren}, {P}h{D} {T}hesis, {U}niversity of {H}amburg, {S}haker
  {V}erlag, {A}achen, {G}ermany, {ISBN} 978-3-8265-8755-9, 2000.

\bibitem{Steenrod51}
N.~Steenrod, \emph{The Topology of Fibre Bundles}, Princeton Mathematical Series \textbf{14}, Princeton University Press, 1951.

\bibitem{WangXu09}
C.~Wang and D.~Xu, \emph{Regularity of Dirac-harmonic maps}, Int. Math. Res. Not. IMRN  2009, no. 20, 3759--3792.

\bibitem{Weidmann}
J.~Weidmann, \emph{Linear operators in Hilbert spaces}, Graduate Texts in Mathematics \textbf{68}, Springer-Verlag, New York-Berlin, 1980.

\bibitem{Xin84}
Y.L.~Xin, \emph{Nonexistence and existence for harmonic maps in Riemannian manifolds}, Proceedings of the 1981 Shanghai symposium on differential geometry and differential equations (Shanghai/Hefei, 1981), 529--538, Science Press, Beijing, 1984. 

\bibitem{Yang09}
L.~Yang, \emph{A structure theorem of Dirac-harmonic maps between spheres}, Calc. Var. Partial Differential Equations \textbf{35} (2009), no. 4, 409--420. 

\bibitem{Zhao07}
L.~Zhao, \emph{Energy identities for Dirac-harmonic maps}, Calc. Var. Partial Differential Equations \textbf{28} (2007), no. 1, 121--138.

\bibitem{Zhu_CVPDE09}
M.~Zhu, \emph{Dirac-harmonic maps from degenerating spin surfaces. I. The Neveu-Schwarz case}, Calc. Var. Partial Differential Equations \textbf{35} (2009), no. 2, 169--189.

\bibitem{Zhu_AGAG09}
\bysame, \emph{Regularity for weakly Dirac-harmonic maps to hypersurfaces}, Ann. Global Anal. Geom. \textbf{35} (2009), no. 4, 405--412. 

\end{thebibliography}

\providecommand{\bysame}{\leavevmode\hbox to3em{\hrulefill}\thinspace}


\end{document}